\documentclass[UTF-8,reqno,12pt]{amsart}
\usepackage{enumerate}
\usepackage{mhequ}
\usepackage{bbm}
\usepackage{geometry}
\usepackage{comment}

\usepackage{enumitem}  % 用于自定义编号列表
\newlist{hypothesis}{enumerate}{1}
\setlist[hypothesis]{label=({\bfseries A}\arabic*),leftmargin=*}  % 设置列表样式

\usepackage{amssymb,url,color, booktabs,nccmath}
\usepackage{mathrsfs}
\usepackage{graphicx}
\usepackage{tikz}
\usepackage{amsthm}

\usepackage{relsize}%强制拉直符号, 使用\mathlarger{\int}

\usepackage{color}
\usepackage[colorlinks=true]{hyperref}
\hypersetup{
	linkcolor=blue,          % color of internal links
	citecolor=red,        % color of links to bibliography
	filecolor=blue,      % color of file links
	urlcolor=cyan
}

\makeatletter
		\renewcommand{\subsection}{\@startsection
			{subsection} % name
			{2} % level
			{0mm} % indent
			{0.5\baselineskip} % befores
			{0.3\baselineskip} % afterskip
			{\normalfont\normalsize\raggedright}} % style
\makeatother
  
\allowdisplaybreaks[4]

\usepackage{fancyhdr} % 用于自定义页眉
\numberwithin{equation}{section}
\def\theequation{\arabic{section}.\arabic{equation}}
\newcommand{\be}{\begin{eqnarray}}
	\newcommand{\ee}{\end{eqnarray}}
\newcommand{\ce}{\begin{eqnarray*}}
	\newcommand{\de}{\end{eqnarray*}}
\newtheorem{theorem}{Theorem}[section]
\newtheorem{lemma}[theorem]{Lemma}
\newtheorem{remark}[theorem]{Remark}
\newtheorem{definition}[theorem]{Definition}
\newtheorem{proposition}[theorem]{Proposition}

\newtheorem*{theorem*}{Theorem}
\newtheorem*{remark*}{Remark}

\def\geq{\geqslant}
\def\leq{\leqslant}

 \def\R{\mathbb R}
 \def\R{\mathbb R}

\def\<{\langle} \def\>{\rangle}
\def\x{{\bf x}}
\def\y{{\bf y}}

\allowdisplaybreaks

\def\E{\mathbb{E}}
\def\P{\mathbb{P}}
\def\R{\mathbb{R}}
\def\S{\mathbb{S}}

\def\d{\mathrm{d}}

\def\1{\mathbbm{1}}
\def\e{\mathrm{e}}

\def\ito{It$\hat{\mathrm{o}}$}
\def\gron{Gr$\ddot{\mathrm {o}}$nwall}

\begin{document}

	\title{Well-posedness and ERGODICITY OF FUNCTIONAL STOCHASTIC  PARTIAL DIFFERENTIAL EQUATIONS WITH MARKOVIAN SWITCHING}

	\author{Fubao Xi$^{\text{a,1}}$, Mingkun Ye$^\text{b}$ and Zuozheng Zhang$^{\text{a,1,*}}$}
	%\address[]{Department of Mathematics, Beijing Institute of Technology, Beijing 100081, China%; Fakult\"at f\"ur Mathematik, Universit\"at Bielefeld, D-33501 Bielefeld, Germany
	%}
    \dedicatory{$^{\rm a}$School of Mathematics and Statistics,
		Beijing Institute of Technology, 
		Beijing, 100081,  P.R.China \\
    $^{\rm b}$School of Mathematics,
		Sun Yat-sen University, Guangzhou, 510275, P.R.China
     }
        \thanks{$^{1}$  Supported in part by the National Natural Science Foundation of China under Grant No. 12071031.}
        \thanks{$^{*}$ Corresponding author.}
        \thanks{E-mail: xifb@bit.edu.cn(F.X.), yemk@mail2.sysu.edu.cn(M.Y.),
    zuozhengzhang@mail.bnu.edu.cn(Z.Z.)}

	\begin{abstract}
 	This work focuses on a class of semi-linear functional stochastic  partial differential equations with Markovian switching, in which the switching component may have finite or countably infinite  states. The well-posedness of the underlying process is obtained by Skorokhod’s representation of the switching component. Then, the exponential mixing of such processes in a finite state space is derived by using the so-called remote start method proposed firstly by Prato and Zabczyk in \cite{Prato_Zabczyk_Ergodicity_1996}. Finally, the corresponding result in a countable infinite state space is further obtained via the finite partition method.\\
	% 摘要后面紧跟的 AMS数学分类
    \\ 
	\end{abstract} 

	%\keywords{} 这样关键词在页脚处
	%\date{\today}
	
	\maketitle

 %在生成的目录中，只显示到节级别，即章节和节，而不会显示子节和子子节
%\setcounter{tocdepth}{2}
 
 %生成目录
%\tableofcontents

\section{Introduction}

Let \( \left(H,\langle\cdot, \cdot\rangle_{H},\|\cdot\|_{H}\right) \) be a real separable Hilbert space, and \( (W(t))_{t \geq 0} \) be a cylindrical Wiener process on another real separable Hilbert space \( \left(U,\langle\cdot, \cdot\rangle_{U}\right) \) under a complete filtered probability space \( \left(\Omega, \mathscr{F},\left(\mathscr{F}_{t}\right)_{t \geq 0}, \mathbb{P}\right) \). 
The space of all Hilbert–Schmidt operators from $U$ to $H$ is denoted by $L_2(U; H)$. 
%Let \( \mathscr{L}(H) \) and \( \mathscr{L}_{L_2}(H) \) be the spaces of all bounded linear operators and Hilbert-Schmidt operators on \( H \), respectively. Denote by \( \|\cdot\| \), , \( \|\cdot\|_{L_2} \) the operator norm, norm on space $H$, and the Hilbert-Schmidt norm, respectively.
%To operate the segment process of path continuous stochastic processes, it is wise to introduce the following function space.
Let $ \mathbb{S} :=\{1,\cdots,N\}$, $1\leq N< \infty$ or $N=\infty$, be the state space.  Let \( \mathscr{C}=C((-\infty, 0] ; H) \) be the space of all continuous functions \( f:(-\infty, 0] \to H \), and
\[
 \mathscr{C}_r:=\left\{\mathbf{x} \in \mathscr{C}:
 \lim_{\theta\to-\infty }\e^{r\theta}\x(\theta) \text{ exists in }H
 \right\}.
\]
with the norm $\|\mathbf{x}\|_r:=\sup_{-\infty<\theta \leqslant 0 }\e^{r\theta}\|\mathbf{x}(\theta)\|_{H}<\infty$. Then $\left(\mathscr{C}_r,\|\cdot\|_r\right)$ is a Polish space; see \cite{hino2006functional} for more details.

In this paper, consider the semi-linear functional stochastic  partial differential equation with Markovian switching (FSPDEwM to be short)
\begin{equation}\label{EQ:FRSPDE:01}
    		\begin{cases}
    		    \d X(t)=\left[A(\Lambda(t)) X(t)+b\left(X_{t}, \Lambda(t)\right)\right] \d t+\sigma\left(X_{t}, \Lambda(t)\right) \d W(t),\\
          \mathbb{P}(\Lambda(t+\Delta)=l \mid \Lambda(t)=k)= \begin{cases}q_{kl} \Delta+o(\Delta), & l \neq k, \\ 1+q_{kk} \Delta+o(\Delta), & l=k,\end{cases}\\
          (X_{0},\Lambda(0))=(\varphi,i)\in\mathscr{C}_r\times \mathbb{S},
    		\end{cases}
\end{equation}
where   $A(k):\mathbb{S}\rightarrow \mathscr{L}(H)$, $b:\mathscr{C}_r\times \mathbb{S} \rightarrow H, $ and $\sigma: \mathscr{C}_{r} \times \mathbb{S} \rightarrow L_2(U;H)$ are measurable mappings associated with the corresponding Borel algebra generated by the concrete topology.  The segment process $(X_t)_{t\geq 0}$ of $(X(t))_{t\geq 0}$ is defined by $X_t(\theta):=X(t+\theta),~\theta\in (-\infty,0]$. $Q=(q_{kl})_{N\times N}$ is the generator of the Markovian switching component $(\Lambda(t))_{t\geq 0}$, where $q_{kl} \geq 0$ is the transition rate from $k$ to $l$, if $l \neq k$; while $q_{kk}=-\sum_{l \neq k} q_{kl}$. Assume that the Markov chain $(\Lambda(t))_{t\geq 0}$ is irreducible and independent of the cylindrical Wiener process $(W(t))_{t\geq 0}$.

The so-called regime-switching diffusion in which continuous dynamics and discrete events coexist has received extensive attention in recent years. Meanwhile, such processes have broad applications in stochastic control and optimization, ecological and biological systems, mathematical finance, risk management; see \cite{bao2016permanence,Mao_Yuan_Book_2006,song2012optimal,Yin2010,zhang2001stock,zhu2009competitive} and references therein. Recently,  the ergodicity of the regime-switching diffusion has drawn a great deal of attention. Cloez and Hairer \cite{cloez2015exponential} study the exponential ergodicity for a Markov process with a continuous component and a discrete component by using a coupling argument and a weak form of the Harris theorem. Shao \cite{SHAO2015SPA}  provides some criteria to
justify the ergodocity of regime-switching diffusion processes in Wasserstein distances based on the theory of the $M$-matrix and the Perron–Frobenius theorem. Shao has also done some other work on ergodicity as well; see \cite{shao2018invariant,shao2013strong} for example. Xi \cite{xi2009asymptotic} and Xi and Zhu \cite{xi2017feller} investigate the exponential ergodicity for regime-switching jump diffusion processes using the Lyapunov function method. 

It is worth pointing out that stochastic processes in these literature do  not depend on  past history.  However, these real-world systems often depend not only on their current state but also on certain past states. In other words, they inherently exhibit aftereffects and delays.  Based on the estimate of exponential functionals of the underlying Markov chain, Bao et~al.   \cite{BAO2023NONLINEAR} and Li and Xi  \cite{li2021convergence} show that there exists a unique invariant probability measure which converges exponentially to its equilibrium under the Wasserstein distance for regime-switching diffusions with finite and infinite delay, respectively.   Nguyen et~al. \cite{nguyen2021stability} establish the same exponential ergodicity for regime-switching diffusions with finite delay under some new conditions. Recently, Zhai and Xi \cite{zhai2024exponential} investigate the exponential ergodicity for several kinds of functional stochastic  differential equations (FSDEs) with Markovian switching and finite delay, including FSDEs, neutral FSDEs and  FSDEs driven by Lévy processes.

%As for ergodicity about SPDEs, there are many progress. Martin Harier, Liu Wei, Prato,Harris theorem, Krylov-Bogoliubov theorem, Weak Harris theorem, asymptotic compiling and generalized coupling, 

For the study of the asymptotic behavior of SPDEs, we provide references on two fundamental methods, i.e., the coupling method and the dissipative method.
The coupling methods are derived from \cite{KUKSIN2001CMP,Masmoudi2002CMP,Mattingly2002CMP}, where the authors study the 2D Navier-Stokes equations.
Subsequently, Hairer \cite{hairer2002PTRF} developes this method into a general framework for studying the ergodicity of parabolic SPDEs  that may lack the strong Feller property. Inspired by Hairer's works, Oleg et~al. \cite{OLEG2020AAP} put forward a generalized coupling method to give general sufficient conditions to derive the exponential ergodicity of the stochastic model, including a variety of nonlinear SPDEs with additive forcing. 
%It is noteworthy to note that Markov processes do not necessarily require the strong Feller property in their context. 

In terms of using dissipative methods to study the asymptotic behavior of solutions to SPDEs or FSPDEs, the existing literature provides a wealth of results. For instance, based on standard Krylov–Bogoliubov procedure, Liu \cite{LIU2009JEE} proves the exponential ergodicity of stochastic evolution equations under dissipativity conditions.  Bao et al.~\cite{BAO204NATMA,Bao_Yin_Yuan_Book_2016} introduce a unified approach, referred to as the remote start method (which is still within the framework of the dissipative method), to establish the existence and uniqueness of the invariant measure and the exponential ergodicity for various types of FSDEs, including FSDEs with variable delays, neutral FSDEs and FSDEs driven by jump processes. Furthermore, this method can also be extended to semi-linear FSPDEs driven by cylindrical Wiener processes or cylindrical $\alpha$-stable processes. Further references on ergodicity of infinite-dimensional systems can be found in the monograph \cite{Prato_Zabczyk_Ergodicity_1996}. 

Inspired by these works, we focus on the FSPDEwM  \eqref{EQ:FRSPDE:01} for practical purposes in the modeling of complex systems. To the best of our knowledge, the study on the ergodicity for such processes  is relatively scarce. The aim of this paper is to investigate the exponential mixing of the underlying process by using the remote start method.  In this paper,   we first give the existence and uniqueness of mild solutions (see Definition \ref{DEf:1103:1} below) for the FSPDEwM  \eqref{EQ:FRSPDE:01} in Section \ref{sec:3}.  For this aim, we study the well-posedness for FSPDEs \eqref{EQ:GENERAL:01} without Markovian switching under the Lipschitz condition and the dissipativity condition, respectively; see Appendix~\ref{Appendix} for more details. Then, based on the above results, we derive the main result in this section by virtue of Skorokhod’s representation \eqref{Eq:1026:1} of the switching process $(\Lambda(t))_{t\geq 0}$ as in \cite{NG2016}. The strong Markov property of the process $(X_t,\Lambda(t))_{t\geq 0}$ is also studied. 

In Section \ref{sec:4}, we use the remote start method to deal with the  FSPDEwM \eqref{EQ:FRSPDE:01}  in a finite state space. Just like the argument in \cite{ Bao_Yin_Yuan_Book_2016}, we need to construct the corresponding double-side processes. It is noteworthy that, due to the existence of the switching process, we also need to construct a double-sided Poisson process, which is given by virtue of Skorokhod's representation. Benefit from the driving process $(\overline{W},N_0)$, we can consider the regularized FSPDEwM \eqref{EQ:FRSPDE:02} which can start from any initial time $s\in \R.$  Then, we shall demonstrate the uniform boundedness of the solution process $(X(t))_{t\geq s}$ and establish the $L^2(\Omega;H)$-convergence of two processes $(X(t;s,(\varphi,i)))_{t\geq s}$ and $(X(t;s,(\psi,i)))_{t\geq s}$ starting from different initial data. Thanks to these careful calculations, we can extend the aforementioned results to the segment process $(X_t)_{t\geq s}$. Finally, we leverage the remote start method to derive the exponential mixing.

In Section \ref{sec:5}, we continue to investigate the exponential mixing for  the FSPDEwM  \eqref{EQ:FRSPDE:01} in an infinite state space. To overcome the difficulty caused by the infinite state space, we use the finite partition approach developed in \cite{SHAO2015SPA}.  More precisely, we divide the infinite state space $\S$ into finite subsets and construct a new Markov chain in a finite state space. Moreover, based on an estimate of the exponential moment of the stopping time $\tau$, we arrive at \eqref{Eq:1019:1}. With the aid of the new Markov chain and  \eqref{Eq:1019:1}, we also obtain the exponential mixing in the infinite state space.

This paper is organized as follows. Section \ref{sec:2} introduces some necessary notations, notions, assumptions together with some basic theories. 
%Section \ref{sec:3} presents the existence and uniqueness of mild solutions for the associated  FSPDEwM  \eqref{EQ:FRSPDE:01}. Besides, the Markov property of the process is also studied. 
Section \ref{sec:3}  establishes the existence and uniqueness of mild solutions for the FSPDEwM  \eqref{EQ:FRSPDE:01}. In addition, the strong Markov property of the process $(X_t,\Lambda(t))_{t\geq 0}$ is also studied. Section \ref{sec:4} is  devoted to establishing the uniform boundedness and convergence for the segment process, and the exponential mixing  for the FSPDEwM  \eqref{EQ:FRSPDE:01} in a finite state space. Section \ref{sec:5} proceeds to the study of  the FSPDEwM  \eqref{EQ:FRSPDE:01} in an infinite state space and derives a similar result.
\section{Preliminary}\label{sec:2}
In this section, we first prepare some notation, assumptions and definitions.  Write $\mathbb{R}^{+}=[0, \infty)$ and $\mathbb{R}^{-}=(-\infty, 0]$.  Let $A$ be a vector or a matrix, we use $A^{\top}$ to denote its transpose.  We say that \( A \geq 0 \) if all elements of \( A \) are non-negative. We denote \( A > 0 \) to indicate that \( A \geq 0 \) and that at least one element of \( A \) is positive. Furthermore, we say \( A \gg 0 \) if all elements of \( A \) are strictly positive. Conversely, \( A \ll 0 \) signifies that \( -A \gg 0 \). 
For $a, b \in \mathbb{R}$, $a \wedge b=\min \{a, b\}$, $a \vee b=\max \{a, b\},$ and $a\lesssim b$ means $a\leq cb$  for some $c>0$. For $\xi=(\xi(1),\cdots,\xi(N))^{\top}$, define $\xi_{max}=\max_{1 \leq k \leq N}\xi(k)$ and $\xi_{min}=\min_{1 \leq k \leq N}\xi(k)$ as $N<\infty$, and $\xi_{sup}=\sup_{k\geq 1}\xi(k)$ and $\xi_{min}=\inf_{k\geq 1}\xi(k)$ as $N=\infty$. Let $\textbf{0}_H$ be the zero element of the space $H$. For any $T\in L_2(U; H),$ define
\[
\|T\|_{L_{2}(U; H)}^{2}:=\sum_{k=1}^{\infty}\left\|T e_{k}\right\|_H^{2},
\]
which does not depend on the choice of the orthonormal basis \( (e_{k})_{k \geq 1} \) on $U$.  For simplicity, we also write \( \|\cdot\|_{L_{2}} \) instead of \( \|\cdot\|_{L_{2}(U; H)} \). Denote by \( \|\cdot\| \) the operator norm. Let $\mathscr{P}(\mathscr{C}_{r} \times \mathbb{S})$ be the set of all probability measures among $\mathscr{C}_{r} \times \mathbb{S}$. Let $\mathscr{P}(\mathbb{R}^{-})$ be the set of all probability measures among $\R^{-},$ and
\[
\mathscr{P}_{r}\left(\mathbb{R}^{-}\right):=\left\{\rho \in \mathscr{P}\left(\mathbb{R}^{-}\right): \rho^{(r)}:=\int_{-\infty}^{0} \e^{-r \theta} \rho(d \theta)<\infty\right\}.
\]

To make our argument in the following more precise, we give an explicit construction of the complete filtered probability space \( \left(\Omega, \mathscr{F},\left(\mathscr{F}_{t}\right)_{t \geq 0}, \mathbb{P}\right) \). Let
	\[
	\Omega_{1}=\left\{\omega\mid \omega:[0, \infty) \rightarrow H \text { is continuous with } \omega(0)=\textbf{0}_H\right\},
	\]
	which is endowed with the locally uniform convergence topology and a probability measure ${\mathbb{P}_{1}}$ so that $\{\langle W(t, \omega),e_k\rangle \}_{k\geq 1}$ is a denumerable sequence of i.i.d. \( \mathbb{R} \)-valued Wiener processes, where ${W(t, \omega):=\omega(t), t \geq 0}$. Put
	\[
	{\Omega_{2}=\{\omega \mid \omega:[0, \infty) \rightarrow \S \text{ is right continuous with left limit}\}}
	\]
	endowed with Skorokhod topology and a probability measure ${\mathbb{P}_{2}}$ so that the coordinate process
	$
	\Lambda(t, \omega):=\omega(t),  t \geq 0,
	$
	is a continuous time Markov chain with the generator ${Q=(q_{ij})_{N\times N}}$. Let
 \[
	(\Omega, \mathcal{F}, \left\{\mathscr{F}_t\right\}_{t \geq 0},\mathbb{P})=(\Omega_{1} \times \Omega_{2},  \mathscr{B}\left(\Omega_{1}\right) \otimes \mathscr{B}\left(\Omega_{2}\right), \{\mathscr{F}_t^{(1)}\otimes\mathscr{F}_t^{(2)}\}_{t \geq 0} ,\mathbb{P}_{1} \times \mathbb{P}_{2}),
	\]
where $\mathscr{B}\left(\Omega_{1}\right),\mathscr{B}\left(\Omega_{2}\right)$  are $\sigma$-algebras generated by the corresponding topology, and $\{\mathscr{F}_t^{(1)}\}_{t\geq 0}$, $\{\mathscr{F}_t^{(2)}\}_{t\geq 0}$ are complete natural filtration generated by $(W(t))_{t\geq 0}$ and $(\Lambda(t))_{t\geq 0}$ respectively.	Hence, under ${\mathbb{P}=\mathbb{P}_{1} \times \mathbb{P}_{2}}$, for ${\omega=\left(\omega_{1}, \omega_{2}\right), t \mapsto \omega_{1}(t)}$ is a cylindrical Wiener process on \( H \), and ${t \mapsto \omega_{2}(t)}$ is a continuous time Markov chain with the generator ${Q=(q_{ij})_{N\times N}}$. Moreover, $(W(t))_{t\geq 0}$ is independent of $(\Lambda(t))_{t\geq 0}$.  Throughout this paper, we will work on this complete filtrated probability space constructed above.

To guarantee the existence and uniqueness of the mild solution for the FSPDEwM \eqref{EQ:FRSPDE:01}, we first assume that for any $k \in \S$, $b(\cdot,k)$ and $\sigma(\cdot,k)$ are measurable and locally Lipschitz throughout this paper. Moreover,  we propose the following assumptions $(\mathbf{A})$:

\noindent(\hypertarget{A1}{A1}) For any \( k \in \mathbb{S} \), $ (-A(k), D(A(k))) $ is a self-adjoint operator with discrete spectrum
\[ 
0<\lambda_{1}(k) \leq \lambda_{2}(k) \leq \cdots \leq \lambda_{n}(k) \leq \cdots
 \]
where multiplicities are counted, and such that $ \lambda_{n}(k) \uparrow \infty, n \rightarrow \infty $. Furthermore, for any $ k\in \mathbb{S}, A(k) $ can generate a $ C_0 $-semigroup \( (\e^{t A{(k)}})_{t \geq 0} \) satisfying \( \|\e^{t A(k)}\| \leqslant \e^{-\lambda_{1}(k) t}, t \geqslant 0 \).

\noindent(\hypertarget{A2}{A2}) For any \(k\in  \mathbb{S} \), there exist constants \( \alpha(k) \in \mathbb{R} \), \( \beta(k)>0 \)  and a probability measure \( \rho \in \mathscr{P}_{2 r}\left(\mathbb{R}^{-}\right) \) such that for any  \( \mathbf{x} , \mathbf{y} \in \mathscr{C}_{r} \), 
	\[ 
	\begin{aligned} 
		& 2\langle\mathbf{x}(0)-\mathbf{y}(0), b(\mathbf{x}, k)-b(\mathbf{y}, k)\rangle_{H}\leqslant  \alpha(k)\|\mathbf{x}(0)-\mathbf{y}(0)\|_{H}^{2}+\beta(k) \int_{-\infty}^{0} \| \mathbf{x}(\theta)-\mathbf{y}(\theta)\|_{H}^{2} \rho(\d \theta).
		\end{aligned} 
	 \]
 \noindent(\hypertarget{A3}{A3}) There exists  a constant \( L>0 \) such that for any  \( \mathbf{x} , \mathbf{y} \in \mathscr{C}_{r} \) and \(k\in  \mathbb{S} \),
	\[
	\left\|\sigma(\mathbf{x}, k)-\sigma(\mathbf{y}, k)\right\|_{L_2}^{2}  
		 \leq L\left(\|\mathbf{x}(0)-\mathbf{y}(0)\|_{H}^{2}+\int_{-\infty}^{0}\|\mathbf{x}(\theta)-\mathbf{y}(\theta)\|_{H}^{2} \rho(\d \theta)\right),
	\] 
    where $\rho$ is determined in (\hyperlink{(A2)}{A2}).
  %\textcolor{red}{\noindent(\hypertarget{(A4)}{A4}) $~ \max _{k\in \mathbb{S}} \sup _{t \geq 0}\left(\mathlarger{\int}_{0}^{t}\|\e^{sA(k)} \sigma_{0}(k)\|_{L_2}^{2} \d s\right)<\infty.$}%\end{hypothesis}
%\end{enumerate}

When $\S$ is an infinite set, we need some additional assumptions $(\mathbf{B})$ as follows.

\noindent(\hypertarget{B1}{B1}) Assume that for any $k\in \S$,
\begin{equation*}
    q_k:=-q_{kk}=\sum_{l\neq k}q_{kl}\leq M,
\end{equation*}
where $M$ is a positive constant independent of $k$.

\noindent(\hypertarget{B2}{B2}) 
Assume that  ${\alpha_{sup}},{\beta_{sup}}<\infty$ and $$\sup_{k\in \S}\|b(\textbf{0}, k)\|_H,\;  \sup_{k\in \S}\|\sigma(\textbf{0}, k)\|_{L_2}<\infty.$$

In what follows, we write ${X(t)}$, $X_t$ and ${\Lambda(t)}$ as ${X(t;0,(\varphi,i))}$, $X_{t}(0,(\varphi,i))$ and ${\Lambda(t;0,i)}$ respectively, to emphasize the initial data $(X_{0},\Lambda(0))=(\varphi,i)$. Define a family of transition probability by
\[
P_t((\varphi,i),\cdot)=\mathbb{P}((X_t(0,(\varphi,i)),\Lambda(t;0,i))\in \cdot)
\]
for any $t\geq 0$ .  By virtue of Theorem \ref{Thm:1105:1}, the process \( (X_{t},\Lambda(t))_{t \geq 0} \) is a strong Markov process. Define the Markovian transition semigroup \( (P_{t})_{t \geq 0} \) associated with the aforementioned process by \[
P_{t} f(\varphi,i):=\mathbb{E} [f (X_{t},\Lambda(t))] =\int_{\mathscr{C}_r\times \mathbb{S}}f(\psi,j)P_t((\varphi,i),\d \psi\times\d j), 
\]
for any $t \geq 0, (\varphi ,i)\in \mathscr{C}_r\times \mathbb{S},$ and $ f \in \mathscr{B}_{b}(\mathscr{C}_r\times \mathbb{S})$. By the Markov property, we have \( P_{t} \circ P_{s}=P_{t+s}, t, s \geq 0 \), where \( P_{t} \circ P_{s} \) means the composition of the operators \( P_{t} \) and \( P_{s} \).

\begin{definition}
    A probability measure \( \mu \in \mathscr{P}(\mathscr{C}_r\times\mathbb{S}) \), the collection of all probability measures on \( \mathscr{C}_r\times \mathbb{S} \), is invariant w.r.t. \( \left(P_{t}\right)_{t \geq 0} \) if
    \begin{equation}\label{Eq:5}
        \mu\left(P_{t} f\right)=\mu(f):=\int_{\mathscr{C}_r\times \mathbb{S}} f(\mathbf{x},k) \mu(\mathrm{d} \mathbf{x},\d k), \quad t \geq 0, f \in \mathscr{B}_{b}(\mathscr{C}_r\times \mathbb{S}) .
    \end{equation}
In this case, the FSPDEwM  \eqref{EQ:FRSPDE:01} is said to admit an invariant probability measure.
\end{definition}
To investigate the exponential convergence of \( \{P_{t}((\varphi,i), \cdot) \}_{t\geq 0}\), let us define a distance \( d(\cdot,\cdot) \) on \( \mathscr{C}_{r} \times \mathbb{S} \),
\[
d((\mathbf{x}, k),(\mathbf{y}, l)):=\|\mathbf{x}-\mathbf{y}\|_{r}+\ell(k, l), \quad(\mathbf{x}, k),(\mathbf{y}, l) \in \mathscr{C}_{r} \times \mathbb{S},
\]
where
\[
\ell(k, l)=\left\{\begin{array}{ll}
1, & k \neq l, \\
0, & k=l,
\end{array}\right.
\]
stands for the standard discrete distance on \( \S \). It is easy to verify that \( (\mathscr{C}_{r} \times \mathbb{S}, d(\cdot, \cdot)) \) is a Polish space.

\begin{definition}
    An  invariant probability measure \( \mu \) for \( \left(P_{t}\right)_{t \geq 0} \) is said to be exponentially mixing with exponent \( \lambda>0 \) and function \( c: \mathscr{C}_r\times\mathbb{S} \mapsto(0, \infty) \) if
\[
\left|P_{t} f(\varphi,i)-\mu(f)\right| \leq c(\varphi,i) \e^{-\lambda t}\|f\|_{\text {Lip }}, \quad t \geq 0, (\varphi,i) \in \mathscr{C}_r\times\mathbb{S}, f \in Lip(\mathscr{C}_r\times\mathbb{S}),
\]
in which \( Lip(\mathscr{C}_r\times\mathbb{S}) \) denotes the set of all Lipschitz functions \( f: \mathscr{C}_r\times\mathbb{S} \mapsto \mathbb{R} \), and 
\[ 
\|f\|_{\text {Lip }}:=  \sup \left\{\frac{|f(\mathbf{x}, k)-f(\y, l)|}{d((\mathbf{x}, k),(\y, l))} ;(\mathbf{x}, k) \neq(\y, l)\right\}<\infty, 
\]
i.e., the Lipschitz constant of \( f \). In other words, the Markovian transition semigroup \( \left(P_{t}\right)_{t \geq 0} \) converges exponentially fast to the equilibrium state in a certain sense.
\end{definition}
%As our criteria on the ergodicity of $\left(X_t, \Lambda(t)\right)_{t\geq 0}$ are related to the theory of $M$-matrices, here we introduce some basic definition and notation of $M$-matrices, and refer the reader to the book [2] for more discussions on this well-studied topic. The theory of $M$-matrices has also been used to study the stability of state-independent regime-switching processes in a finite state space (see [14, Theorem 5.3]).

  %Let $B$ be a matrix or vector. By $B \geq 0$ we mean that all elements of $B$ are non-negative. By $B>0$ we mean that $B \geq 0$ and at least one element of $B$ is positive. By $B \gg 0$, we mean that all the elements of $B$ are positive. $B \ll 0$ means that $-B \gg 0$.
  
To proceed, since our criteria for the ergodicity of the process \((X_t, \Lambda(t))_{t \geq 0}\) are closely linked to the theory of \(M\)-matrices, we will introduce some fundamental definitions and notation regarding \(M\)-matrices. For a more comprehensive discussion of this well-established topic, we direct the reader to the book \cite{berman1994nonnegative}.  There are several conditions that are equivalent to the assertion that \( A \) is a non-singular \( M \)-matrix. In the following, we give the definition of \( M \)-matrix and outline some of these conditions; see \cite[Section 2.6]{Mao_Yuan_Book_2006} for further details. 
\begin{definition}[$M$-Matrix]
		A square matrix $A=\left(a_{i j}\right)_{n \times n}$ is called an $M$-Matrix if $A$ can be expressed in the form $A=s I-B$ with some $B \geq 0$ and $s \geq \operatorname{Ria}(B)$, where $I$ is the $n \times n$ identity matrix, and $\operatorname{Ria}(B)$ the spectral radius of $B$. When $s>\operatorname{Ria}(B), A$ is called a non-singular $M$-matrix.
	\end{definition}
		
%There are many conditions which are equivalent to the statement that $A$ is a non-singular $M$-matrix, see \cite[Section 2.4]{2006Markovian} or \cite[Chapter 6]{1979Matrix} for more details. We quote some of them below.
%There are several conditions that are equivalent to the assertion that \( A \) is a non-singular \( M \)-matrix. For further details, refer to \cite[Section 2.6]{Mao_Yuan_Book_2006}  . Below, we list some of these conditions.
\begin{proposition}\label{prop:aini}
		The following statements are equivalent.
		\begin{enumerate}
			\item  $A$ is a non-singular $n \times n$ $M$-matrix.
			\item  Every real eigenvalue of $A$ is positive.
			\item  $A$ is semipositive; that is, there exists $x \gg 0$ in $\mathbb{R}^n$ such that $A x \gg 0$.
			\item  $A$ is inverse-positive; that is, $A^{-1}$ exists and $A^{-1}\geq 0.$
			\item  For any $y\gg 0$ in $\mathbb{R}^n$, the linear equation $Ax=y$ has a unique solution $x\gg 0$.
		\end{enumerate}
	\end{proposition}

\section{Well-posedness and strong Markov property }\label{sec:3}
According to the general definition of mild solution (see \text{ \cite[Definition 6.2.1]{Liu_Rockner_Book_2015}}), we give the following definition of mild solutions with respect to the FSPDEwM   \eqref{EQ:FRSPDE:01}.  For this aim,  we first give the definition of random semigroup. 
\begin{definition}
Let $E$ be a Banach space and \( Y(E) \) be the topological vector space consisting of strongly continuous mappings from \([0, \infty) \times [0, \infty)\) into the Banach space \( B(E) \), which is comprised of all  linear transformations on the Banach space \( E \). A random non-time-homogeneous semigroup is defined as a random variable taking values in \( S(E) \subset Y(E) \), where \( S(E) \) represents the set of strongly continuous two-parameter semigroups of operators on \( E \).
\end{definition}

\begin{definition}\label{DEf:1103:1}
An \( H \)-valued predictable process \( X(t), t \in[0, T] \), is called a mild solution of  the FSPDEwM \eqref{EQ:FRSPDE:01}  if
\begin{equation}\label{EQ:MILD:12}
    \begin{aligned}
X(t)&=S(0,t,\Lambda(\cdot)) \varphi(0)  +\int_{0}^{t} S(s,t,\Lambda(\cdot)) b(X_s,\Lambda(s)) \mathrm{d} s \\
 &\quad+\int_{0}^{t} S(s,t,\Lambda(\cdot))  \sigma(X_s,\Lambda(s)) \mathrm{d} W(s) \quad \P \text {-a.s.}
\end{aligned}
\end{equation}
for each \( t \in[0,\infty) \), where 
$$S(s,t,\Lambda(\cdot)):=\exp\left\{\int_{s}^{t}A(\Lambda(u))\d u\right\}$$
is the random  non-time-homogeneous semigroup with respect to operators $A(\Lambda(\cdot))$ on \( H \). In particular, the integrals appearing in \eqref{EQ:MILD:12} have to be well defined.
\end{definition}
To proceed, we introduce Skorokhod’s representation of the switching process $(\Lambda(t))_{t\geq 0}$ in terms of the Poisson random measure as in \cite{NG2016}. Precisely, let
\[
	\Delta_{1 2}=\left[0, q_{1 2}\right), \Delta_{1 l}=\left[\sum_{j=2}^{l-1}q_{1j},\sum_{j=2}^{l}q_{1j}\right),\quad l\geq 3,
	\]
and for each $k \in \mathbb{S}$ and $k>1$, let
\[
	\Delta_{k 1}=\left[0, q_{k 1}\right), \Delta_{k l}=\left[\sum_{j=1,j\neq k}^{l-1}q_{kj},\sum_{j=1,j\neq k}^{l}q_{kj}\right),\quad l>1,l\neq k.
	\]
 Let
	\[
	U_{k}=\bigcup_{l \geq 1, l \neq k} \Delta_{kl}, \quad k \geq 1.
	\]
	For notation convenience, we put \( \Delta_{kk}=\varnothing \) and \( \Delta_{kl}=\varnothing \) if \( q_{kl}=0, k\neq l \). Note that for each $k \in \mathbb{S}$, $\{\Delta_{kl}:l\in \mathbb{S}\}$ are disjoint intervals, and the length of the interval $\Delta_{kl}$ is equal to $q_{kl}$.  When $\S$ is a finite set, it is obvious that  $\mathfrak{m}\left(U_{k}\right)$ is bounded  by $\max_{1\leq l\leq N}q_{l}$  for any $k\in\S$,  where \( \mathfrak{m}(\mathrm{d} x) \) denotes the Lebesgue measure over \( \mathbb{R} \). When $\S$ is an  infinite set,  (\hyperlink{B1}{B1}) ensures that $\mathfrak{m}\left(U_{k}\right)$ is also bounded. Without loss generality, we assume that $\mathfrak{m}\left(U_{k}\right)$ is bounded above by $M$, whether $\mathbb{S}$ is a finite set or not.

	Let \( \xi_{n}^{(k)}, k, n=1,2, \ldots \), be \( U_{k} \)-valued random variables with
	$\mathbb{P}(\xi_{n}^{(k)} \in \mathrm{d} x)=\mathfrak{m}(\d x)/\mathfrak{m}(U_{k}),
	$ and \( \tau_{n}^{(k)}, k, n \geq 1 \), be non-negative random variables satisfying \( \mathbb{P}(\tau_{n}^{(k)}>t)=\exp\{-\mathfrak{m}\left(U_{k}\right)t\} \), \( t \geq 0 \). Suppose that \( \{\xi_{n}^{(k)}, \tau_{n}^{(k)}\}_{k, n \geq 1} \) are all mutually independent. Put
	\[
	\zeta_{0}^{(k)}=0, \quad k \geq 1;\quad \zeta_{n}^{(k)}=\tau_{1}^{(k)}+\cdots+\tau_{n}^{(k)}, \quad k,n  \geq 1.
	\]
	Let
	\[
	D_{p}=\bigcup_{k \geq 1} \bigcup_{n \geq 0}\left\{\zeta_{n}^{(k)}\right\}\text{ and }p(\zeta_{n}^{(k)})=\xi_{n}^{(k)}\quad k,n  \geq 1.
	\]
	Correspondingly, put	\begin{equation}\label{eq:bujini}
		N([0, t] \times A)=\#\left\{ s \in D_{p}:0< s \leq t, p(s) \in A\right\}, t>0, A \in \mathscr{B}([0, \infty)) .
	\end{equation}
	As a consequence, we derive a Poisson point process \( (p(t))_{t\geq 0} \) and a Poisson random measure \( {N}(\mathrm{d} t, \mathrm{d} u) \) with intensity \( \mathfrak{m}(\mathrm{d} u)\mathrm{d} t  \). Moreover, we know that \( \mathfrak{m}(\d u) \d t\) is the compensator of  $N(\d t, \d u)$. Define a function $h: \mathbb{S} \times\left[0, M\right]\to \mathbb{R}$ by
	\begin{equation}\label{EQ:0916:01}
	    h(k, u)=\sum_{l \in \mathbb{S}}(l-k) \1_{\triangle_{k l}}(u).
	\end{equation}
 Then, $(\Lambda(t))_{t\geq 0}$ can be reformulated by the following SDE,
 \begin{equation}\label{Eq:1026:1}
     \begin{split}
		\d \Lambda(t) & =\int_{[0, M]} h(\Lambda(t-), u) N(\d t, \d u). \\
\end{split}
 \end{equation}

\begin{theorem}\label{Eq:1026:2}
    Let (\hyperlink{(A1)}{A1})-(\hyperlink{(A3)}{A3}) hold. If   
 $\S$ is an infinite set, we need to additionally assume that  (\hyperlink{B1}{B1}) holds. Then for any initial data $(\varphi,i)\in \mathscr{C}_r\times \S$, the FSPDEwM \eqref{EQ:FRSPDE:01}  admits a unique  mild solution. 
\end{theorem}
\begin{proof}
    We use the successive construction method to prove the existence and uniqueness of mild solutions.  Recall that $\mathfrak{m}\left(U_{k}\right)$ is  bounded. Hence, almost every sample path of $(\Lambda(t))_{t\geq 0}$ is a right continuous step function with a finite number of simple jumps on $[0,M]$. So there is a family of non-decreasing stopping times $\{\tau_n\}_{n\geq 0}$ defined by
\begin{equation*}
        \tau_0:=0,\;
        \tau_{n+1}:=\inf\{t\geq \tau_{n}: \Lambda(t)\neq \Lambda(\tau_{n})\} \quad \text { for any } n \geq 0.
\end{equation*}
We can see that $\Lambda(t)$ is a random constant on every interval $[\tau_n,\tau_{n+1}),$  that is, for every $n\geq 0$, 
\[
\Lambda(t)=\Lambda(\tau_n),\quad \tau_n\leq t<\tau_{n+1}.
\]
First, we consider the FSPDEwM \eqref{EQ:FRSPDE:01} on $t\in [0,\tau_1)$, which becomes
\begin{equation}\label{EQ:MILD:14}
   \begin{cases}
        \d X(t)=\left[A(i) X(t)+b\left(X_{t}, i\right)\right] \d t+\sigma\left(X_{t}, i\right) \d W(t).\\
        X_0=\varphi\in\mathscr{C}_r.
   \end{cases}
\end{equation}
It follows from Lemma \ref{Lem:1024:1} that Eq.\ \eqref{EQ:MILD:14}  admits a unique mild solution, i.e., there exists a unique continuous adapted process $X^{(0)}(t)$ for $t\in[0,\tau_1)$ on $H$ such that 
\begin{equation*}
    \begin{aligned}
X^{(0)}(t)&=\e^{tA(i)}\varphi(0)  +\int_{0}^{t} \e^{(t-s)A(i)} b(X_s^{(0)},i) \mathrm{d} s +\int_{0}^{t} \e^{(t-s)A(i)} \sigma(X_s^{(0)},i) \mathrm{d} W(s) \quad \P \text {-a.s.}
\end{aligned}
\end{equation*}
Next, consider the FSPDEwM \eqref{EQ:FRSPDE:01} on \( t \in [\tau_{1}, \tau_{2})\), which becomes
\begin{equation}\label{EQ:MILD:15}
        \d X(t)=\left[A(\Lambda(\tau_1)) X(t)+b\left(X_{t}, \Lambda(\tau_1)\right)\right] \d t+\sigma\left(X_{t}, \Lambda(\tau_1)\right) \d W(t).
\end{equation}
Similarly, Eq.\ \eqref{EQ:MILD:15} admits a unique mild solution $X^{(1)}(t)$  for \( t \in\left[\tau_{1}, \tau_{2}\right) \), such that 
\begin{equation*}
    \begin{aligned}
X^{(1)}(t)&=\e^{(t-\tau_1)A(\Lambda(\tau_1))}X^{(0)}(\tau_1)  +\int_{\tau_1}^{t} \e^{(t-s)A(\Lambda(\tau_1))} b(X_s^{(1)},\Lambda(\tau_1)) \mathrm{d} s \\
 &\quad+\int_{\tau_1}^{t} \e^{(t-s)A(\Lambda(\tau_1))} \sigma(X_s^{(1)},\Lambda(\tau_1)) \mathrm{d} W(s) \quad \P \text {-a.s.}
\end{aligned}
\end{equation*}
Repeating this procedure,  consider the FSPDEwM \eqref{EQ:FRSPDE:01} on \( t \in [\tau_{n}, \tau_{n+1}\) ), which becomes
\begin{equation}\label{EQ:MILD:15555}
        \d X(t)=\left[A(\Lambda(\tau_n)) X(t)+b\left(X_{t}, \Lambda(\tau_n)\right)\right] \d t+\sigma\left(X_{t}, \Lambda(\tau_n)\right) \d W(t).
\end{equation}
Eq.\ \eqref{EQ:MILD:15555} admits a unique mild solution $X^{(n)}(t)$  for \( t \in\left[\tau_{n}, \tau_{n+1}\right) \), such that 
\begin{equation*}
    \begin{aligned}
X^{(n)}(t)&=\e^{(t-\tau_n)A(\Lambda(\tau_n))}X^{(n-1)}(\tau_n)  +\int_{\tau_n}^{t} \e^{(t-s)A(\Lambda(\tau_n))} b(X_s^{(n)},\Lambda(\tau_n)) \mathrm{d} s \\
 &\quad+\int_{\tau_n}^{t} \e^{(t-s)A(\Lambda(\tau_n))} \sigma(X_s^{(n)},\Lambda(\tau_n)) \mathrm{d} W(s) \quad \P \text {-a.s.}
\end{aligned}
\end{equation*}
Then the mild solution $X(t)$ of the FSPDEwM \eqref{EQ:FRSPDE:01} can be expressed by 
\begin{equation*}
  X(t)= X^{(n)}(t), \quad \text{ if } t\in   [\tau_{n}, \tau_{n+1}).
\end{equation*}
Thus, $X (t)$ is determined uniquely on the time interval $[0, \tau_n )$ for every $n$.  By the continuity of $X^{(n)}(t)$ and the definition of $X(t)$, it is obvious that $X(t)$ is  continuous almost surely on $[0,\tau_{\infty})$, where $\tau_{\infty}:=\lim_{n\to\infty}\tau_n$. In what follows, we shall show that $X(t)$ satisfies Eq. \eqref{EQ:MILD:12}. Actually, for any $ t\in   [\tau_{n}, \tau_{n+1})$, we have
    \begin{align*}
        X(t)&=X^{(n)}(t)\\
        &=\e^{(t-\tau_n)A(\Lambda(\tau_n))}\Big(\e^{(\tau_n-\tau_{n-1})A(\Lambda(\tau_{n-1}))}X^{(n-2)}(\tau_{n-1})  \\
 &\qquad +\int_{\tau_{n-1}}^{\tau_n} \e^{(\tau_n-s)A(\Lambda(\tau_{n-1}))} b(X_s^{(n-1)},\Lambda(\tau_{n-1})) \mathrm{d} s\\
 &\qquad+\int_{\tau_{n-1}}^{\tau_n} \e^{(\tau_n-s)A(\Lambda(\tau_{n-1}))} \sigma(X_s^{(n-1)},\Lambda(\tau_{n-1})) \mathrm{d} W(s) \Big)\\
 &\qquad+\int_{\tau_n}^{t} \e^{(t-s)A(\Lambda(\tau_n))} b(X_s^{(n)},\Lambda(\tau_n)) \mathrm{d} s +\int_{\tau_n}^{t} \e^{(t-s)A(\Lambda(\tau_n))} \sigma(X_s^{(n)},\Lambda(\tau_n)) \mathrm{d} W(s) \\
 &=\e^{\int_{\tau_{n-1}}^{t}A(\Lambda(\tau_r))\d r}X^{(n-2)}(\tau_{n-1})+\int_{\tau_{n-1}}^{t} \e^{\int_s^tA(\Lambda(r))\d r} b(X_s,\Lambda(s)) \mathrm{d} s \\
 &\qquad+\int_{\tau_{n-1}}^{t} \e^{\int_s^tA(\Lambda(r))\d r} \sigma(X_s,\Lambda(s)) \mathrm{d} W(s).
    \end{align*}
Repeating this procedure implies Eq. \eqref{EQ:MILD:12}.

Finally, we need to show $X (t)$ is a global mild solution. It is sufficient to prove that $\tau_{\infty}= \infty$,  $\P$-a.s. When $\S$ is a finite set, it is fulfilled by \cite[Proposition 2.1]{xi2009asymptotic}. When $\S$ is an infinite set, by  (\hyperlink{B1}{B1}), for any $T>0$,
	$$
	\begin{aligned}
		\mathbb{P}\left\{\tau_n \leq T\right\} &=\mathbb{P}\left\{\int_0^{T \wedge \tau_n} \int_{[0,M]} \1_{\left\{u \in\left[0, q_{\Lambda(s-)}\right)\right\}} N(\d s, \d u)=n\right\} \\
		& \leq \mathbb{P}\left\{\int_0^T \int_{[0,M]}  N(\d s, \d u) \geq n\right\} \\
		&=\sum_{k=n}^{\infty} \e^{-M T} \frac{(M T)^k}{k !} .
	\end{aligned}
	$$
It follows that $\mathbb{P}\left\{\tau_n \leq T\right\} \rightarrow 0$ as $n \rightarrow \infty$. As a result, we get our desired assertion. 
\end{proof}
Finally, we conclude this section with the following theorem.
\begin{theorem}\label{Thm:1105:1}
 \( \left(X_{t},\Lambda(t)\right)_{t \geq 0} \) is a time-homogeneous strong Markov process.
\end{theorem}
 \begin{proof}
Let us divide this proof into two steps.

\noindent Step 1 (Time-homogeneity): Let $\mathscr{B}(\mathscr{C}_r)$ denote the Borel $\sigma$-algebra generated by all open sets on $\mathscr{C}_r$.  According to the definition of transition probability, we have that for any $t,u\geq0$, \( B \in \mathscr{B}(\mathscr{C}_r)\) and $j\in\mathbb{S}$
\[
P_{u,t+u}((\varphi,i), B\times \{j\})=\mathbb{P}\left((X_{t+u}(u, (\varphi,i)),\Lambda(t+u;u,i) )\in B\times \{j\}\right),
\]
where \( (X_{t+u}(u, (\varphi,i)),\Lambda(t+u;u,i) )\) is determined by 
\begin{equation*}
    \begin{aligned}
X(t+u)&=\exp\left\{\int_{u}^{t+u}A(\Lambda(r))\d r\right\}\varphi(0)  +\int_{u}^{t+u} \exp\left\{\int_{s}^{t+u}A(\Lambda(r))\d r\right\}b(X_s,\Lambda(s)) \mathrm{d} s \\
 &\quad+\int_{u}^{t+u} \exp\left\{\int_{s}^{t+u}A(\Lambda(r))\d r\right\}  \sigma(X_s,\Lambda(s)) \mathrm{d} W(s),
\end{aligned}
\end{equation*}
and
\[
\Lambda(t+u)=i+\int_{u}^{t+u} \int_{[0,M]} h(\Lambda(s-), z) N(\d s, \d z).
\]
This is equivalent to
\begin{equation*}
    \begin{aligned}
X&(t+u)\\
&=\exp\left\{\int_{0}^{t}A(\Lambda(r+u))\d r\right\}\varphi(0)  +\int_{0}^{t} \exp\left\{\int_{s}^{t}A(\Lambda(r+u))\d r\right\}b(X_{s+u},\Lambda(s+u)) \mathrm{d} s \\
 &\quad+\int_{0}^{t} \exp\left\{\int_{s}^{t}A(\Lambda(r+u))\d r\right\}  \sigma(X_{s+u},\Lambda(s+u)) \mathrm{d} \widehat{W}(s),
\end{aligned}
\end{equation*}
and
\[
\Lambda(t+u)=i+\int_{0}^{t} \int_{[0,M]} h(\Lambda((s+u)-), z) \widehat{N}(\d s, \d z),
\]
where $\widehat{W}(s)=W(s+u)-W(u)$ and $\widehat{N}(s,U)= N(s+u, U)-N(u, U)$.  Note that $X_{t+u}$ completely depends on $X(t + u)$ and its history. Hence, by the definition of $(X_t(0,(\varphi,i)),\Lambda(t;0,i))$, the weak uniqueness implies  \( (X_{t+u}(u, (\varphi,i)),\Lambda(t+u;u,i) )\) and \((X_t(0,(\varphi,i)),\Lambda(t;0,i)) \) are identical in distribution. Consequently,
\[
\mathbb{P}\left((X_{t+u}(u, (\varphi,i)),\Lambda(t+u;u,i) )\in B\times \{j\}\right)=\mathbb{P}\left((X_{t}(0, (\varphi,i)),\Lambda(t;0,i) )\in B\times \{j\}\right),
\]
namely,
\[
P_{u,t+u}((\varphi,i), B\times \{j\})=P_{0,t}((\varphi,i), B\times \{j\}).
\]
\noindent Step 2 (Strong Markov property)  Let \( \tau>0 \) be any stopping time that is finite a.s. Note that \(( W(t) )\) and $(N(t,U))$ are strong Markov processes with independent increment. 
Let \( \mathscr{G}_{\tau}=\sigma\{W(\tau+s)-W(\tau), N(\tau+s, U)-N(\tau, U): s \geq 0, U \in \mathscr{B}([0,M])\} \). Clearly, \( \mathscr{G}_{\tau} \) is independent of \( \mathscr{F}_{\tau} \). 

Write $(X_t(0,(\varphi,i)),\Lambda(t;0,i))=(X_t,\Lambda(t))$ simply. According to the definition of \( X_{t} \), we have \( X_{t}(\theta)=X\big(t+\theta \big)\) 
for any  $\theta\leq 0$.  Moreover, for any \( t>0 \),
\begin{equation*}
    \begin{aligned}
X(t+\tau)&=S(\tau,t+\tau,\Lambda(\cdot)) X(\tau)  +\int_{\tau}^{t+\tau} S(s,t+\tau,\Lambda(\cdot)) b(X_s,\Lambda(s)) \mathrm{d} s \\
 &\quad+\int_{\tau}^{t+\tau} S(s,t+\tau,\Lambda(\cdot))  \sigma(X_s,\Lambda(s)) \mathrm{d} W(s),
\end{aligned}
\end{equation*}
and
\[
\Lambda(t+\tau)=\Lambda(\tau)+\int_{\tau}^{t+\tau} \int_{[0,M]} h(\Lambda(s-), u) N(\d s, \d u).
\]
Hence,  we obtain that \( \big(X_{t+\tau},\Lambda(t+\tau)\big)=\big(X_{t+\tau}(\tau,(X_{\tau},\Lambda(\tau))),\Lambda(t+\tau;\tau,\Lambda(\tau))\big) \).  Note that \(( X_{t+\tau}(\tau, (\x,k) ),\Lambda(t+\tau;\tau,k))\) depends completely on the increments \( W(\tau+s)-W(\tau), N(\tau+s, U)-N(\tau, U) \) and so is \( \mathscr{G}_{\tau} \)-measurable when \( (X_{\tau},\Lambda(\tau))=(\x,k) \) is given.  Now, for any \( B\in \mathscr{B}(\mathscr{C}_r)\) and $j\in\mathbb{S}$, we compute
\begin{align*}
    \mathbb{P}&\Big(\big(X_{t+\tau},\Lambda(t+\tau)\big) \in B\times\{j\} \mid \mathscr{F}_{\tau}\Big)\\
    &=\mathbb{E}\Big(\1_{B\times\{j\}}\big(X_{t+\tau},\Lambda(t+\tau)\big)\mid \mathscr{F}_{\tau}\Big) \\
&=\mathbb{E}\Big(\1_{B\times\{j\}}\big(X_{t+\tau}(\tau,(X_{\tau},\Lambda(\tau))),\Lambda(t+\tau;\tau,\Lambda(\tau))\big)  \mid \mathscr{F}_{\tau}\Big) \\
&=\mathbb{E}\Big(\1_{B\times\{j\}}\big(X_{t+\tau}(\tau,(\x,k))),\Lambda(t+\tau;\tau,k)\big)  \Big)|_{(\x,k)=(X_{\tau}, \Lambda(\tau))} \\
&=\mathbb{P}\Big(\big(X_{t+\tau}(\tau,(\x,k))),\Lambda(t+\tau;\tau,k)\big)\in B\times \{j\}\Big)|_{(\x,k)=(X_{\tau}, \Lambda(\tau))}\\
&=P_t((X_{\tau}, \Lambda(\tau)), B\times\{j\}),
\end{align*}
%By standard monotone class argument and taking conditional expectation with respect to the $\sigma$-algebra generated by $(X_{\tau},\Lambda(\tau))$ implies the Strong Markov property.
where we have used the time-homogeneity in the last equality. By the standard monotone class theorem and taking the conditional expectation with respect to the $\sigma$-algebra generated by $(X_{\tau}, \Lambda(\tau))$, the strong Markov property follows. The proof is complete.
\end{proof}
\begin{remark}
    It is worth pointing out that since $S(s,t,\Lambda(\cdot))$ is a random non-time-homogeneous semigroup, we cannot obtain the result that $(X_t)_{t\geq 0}$ is a time-homogeneous process. However, here we consider the system $(X_t,\Lambda(t))_{t\geq 0}$, whose time-homogeneity is derived from the weak uniqueness of mild solutions.
\end{remark}

\section{Case I: finite state space}\label{sec:4}
In this section, we consider the FSPDEwM \eqref{EQ:FRSPDE:01}  in a finite state space.  Just as stated in the abstract, we shall adopt the remote start method to prove Theorem \ref{THM:MAIN:01}. For this aim, it is necessary to appropriately construct a double-sided cylindrical Wiener process and a double-sided Poisson process.  To do so, we first give a  representation for the cylindrical Wiener process $(W(t))_{t\geq 0}$.  Let \( (e_{k})_{k \geq 1} \) be an orthonormal basis of  \( U\), and \( \left.\alpha_{k} \in( 0, \infty), k \geq 1\right. \), such that \( \sum_{k=1}^{\infty} \alpha_{k}^{2}<\infty \).  Define  the operator \( J: U\rightarrow U \) by
\[
J(u):=\sum_{k=1}^{\infty} \alpha_{k}\left\langle u, e_{k}\right\rangle_{U} e_{k}, \quad u \in U.
\]
% Then by \cite[Remark 2.5.1]{Liu_Rockner_Book_2015}, \( J \) is one-to-one and Hilbert-Schmidt such that 
%\[
%W(t)=\sum_{k=1}^{\infty} B_{k}(t) Je_{k}, \quad t \geq 0
%\]
%for a denumerable sequence of i.i.d (independent and identically distributed) \( \mathbb{R} \)-valued Wiener processes \(\{B_{k}(t)\}_{k \geq 1} \) on \( (\Omega, \mathscr{F},\left(\mathscr{F}_{t}\right)_{t \geq 0}, \mathbb{P}) \). 
According to \cite[Remark 2.5.1]{Liu_Rockner_Book_2015}, the operator \( J \) is  one-to-one and Hilbert-Schmidt. Moreover, $(W(t))_{t\geq 0}$ has a representation
\[
W(t) = \sum_{k=1}^{\infty} B_{k}(t) J(e_{k}), \quad t \geq 0,
\]
where \(\{(B_{k}(t))_{t\geq 0}\}_{k \geq 1}\) is a denumerable sequence of independent and identically distributed \(\mathbb{R}\)-valued standard Wiener processes defined on the probability space \((\Omega, \mathscr{F}, \left(\mathscr{F}_{t}\right)_{t \geq 0}, \mathbb{P})\).

Then, we construct the double-sided cylindrical Wiener process. Let $ \{(\widetilde{B}_k(t))_{t\geq 0}\}_{k\geq 1} $ be an independent copy of $ \{(B_k(t))_{t\geq 0}\}_{k\geq 1} $, and $\widetilde{W}(t)=\sum_{k=1}^{\infty} \widetilde{B}_{k}(t) J(e_{k}).$ Define  $(\overline{W}(t))_{t\in\mathbb{R}}$ as the following double-sided cylindrical Wiener process, 
\[
\overline{W}(t):=\left\{\begin{array}{ll}
	W(t), & t \geq 0, \\
	\widetilde{W}(-t), & t<0,
\end{array}\right.
\]
with its filtration 
\[
\overline{\mathscr{F}}_{t}:=\bigcap_{s>t} \overline{\mathscr{F}}_{s}^{0},
\]
where \( \overline{\mathscr{F}}_{s}^{0}:=\sigma\left(\left\{\overline{W}\left(r_{2}\right)-\overline{W}\left(r_{1}\right):-\infty<r_{1} \leq r_{2} \leq s\right\}, \mathscr{N}_1\right) \) and \( \mathscr{N}_1:=\{A \in \) \( \mathscr{F}_1 \mid \mathbb{P}_1(A)=0\} \).

Next, we use Skorokhod's representation  \eqref{Eq:1026:1} of $(\Lambda(t))_{t\geq 0}$  to construct the double-sided Poisson process.  Let \( \overline{N}(\cdot, \cdot) \) be an independent copy of \( N(\cdot, \cdot) \), and let \( N_{0}(\cdot, \cdot) \) be a double-sided Poisson process defined as
\[
N_{0}(t, \Gamma):=\left\{\begin{array}{ll}
	N(t, \Gamma), & t \geq 0, \\
	\overline{N}(-t, \Gamma), & t<0
\end{array}\right.
\]
for any \( \Gamma \in \mathscr{B}({[0,M]}) \) with its filtration 
\[
\overline{\mathscr{G}}_{t}:=\bigcap_{s>t} \overline{\mathscr{G}}_{s}^{0},
\]
where \( \overline{\mathscr{G}}_{s}^{0}:=\sigma\left(\left\{N_{0}\left(r_{2}, \Gamma\right)-N_{0}\left(r_{1}, \Gamma\right):-\infty<r_{1} \leq r_{2} \leq s, \Gamma\right\}, \mathscr{N}_2\right) \) and \( \mathscr{N}_2:=\{B \in \mathscr{F}_2 \mid \mathbb{P}_2(B)=0\} \).

Let $\mathscr{H}_t = \overline{\mathscr{F}}_{t}\bigotimes\overline{\mathscr{G}}_{t}, t\in \R.$ We shall work on this complete filtrated probability space $(\Omega,\mathscr{H},(\mathscr{H}_t)_{t\in \R},\P)$ in the rest of this paper.  Furthermore, for any fixed \( s \in \mathbb{R} \) and \( (\varphi,i) \in \mathscr{C}_r\times\mathbb{S} \), consider the following semi-linear FSPDEwM:
\begin{equation}\label{EQ:FRSPDE:02}
	\left\{\begin{array}{ll}
		{\d X(t)=\left\{A({\Lambda}(t)) X(t)+b\left(X_{t},{ \Lambda}(t)\right)\right\} \d t+\sigma\left(X_{t}, {\Lambda}(t)\right) \d \overline{W}(t), \quad X_{s}=\varphi ,} \\ 
		{\d  {\Lambda}(t)  =\mathlarger{\int}_{[0, M]} h( {\Lambda}(t-), u) N_0(\d t, \d u) } ,\quad  {\Lambda}(s) =i, \quad t\geq s.
	\end{array}\right.
\end{equation}
By a similar argument as in Theorem \ref{Eq:1026:2}, we can show that under (\hyperlink{(A1)}{A1})-(\hyperlink{(A3)}{A3}), the FSPDEwM \eqref{EQ:FRSPDE:02} admits a unique mild solution $(X (t ; s , (\varphi ,i)),\Lambda(t;s,i))_{t\geq s}$  with the initial data $(X_s,\Lambda(s) )= (\varphi,i)$. 

To prove the main theorem of this section, we proceed as follows. First, we establish the uniform boundedness of solution processes $(X(t))_{t\geq s}$ and the convergence for two processes $(X(t;s,(\varphi,i)))_{t\geq s}$ and $(X(t;s,(\psi,i)))_{t\geq s}$ from different initial data in the sense of $L^2(\Omega;H)$ in Subsection \ref{subsect:4.1}. Then, we further give the same results for the segment process $(X_t)_{t\geq s}$ in Subsection \ref{subsect:4.2}. Finally, we obtain exponential mixing by adopting the remote start method in Subsection \ref{subsect:4.3}. 
\subsection{Boundedness and Convergence of $X(t)$}\label{subsect:4.1}
\begin{lemma}\label{lem:yizhiyoujie}
Let (\hyperlink{(A1)}{A1})-(\hyperlink{(A3)}{A3}) hold.  Assume further that $$\mathcal{A}:=-\big(Q+\operatorname{diag}\left(\alpha(1)-2\lambda_1(1)+L, \ldots, \alpha(N)-2\lambda_1(N)+L\right)\big)$$ is a non-singular $M$-matrix, so $\xi=(\xi(1),\cdots,\xi(N))^{\top}:=\mathcal{A} ^{-1}\vec{1}\gg 0$, where $\vec{1}=(1,\cdots,1)^{\top}$.  If $$(\beta(k)+L)\xi(k) <1,\quad \text{for all } 1\leq k\leq N,$$ then, there exists a positive constant $\lambda$ such that  for any initial data $(\varphi,i) \in \mathscr{C}_{r}\times \S$ and $-\infty<s\leq t <\infty$,  
\begin{equation}\label{Eq:1015:1}
    \mathbb{E}[\left\|X(t;s,(\varphi,i))\right\|_{H}^{2}] \lesssim 1+\e^{-\lambda(t-s)}\|\varphi\|_{r}^{2}.
\end{equation}
\end{lemma}
\begin{proof}
In the sequel,  write $X(t;s,(\varphi,i))=X(t)$   simply.  For any $\theta \in(-\infty, 0]$, let $\textbf{0}(\theta) \equiv \mathbf{0}_H$.   Hence, it follows from \eqref{Eq:1},  \eqref{eq:shoushi},  (\hyperlink{(A2)}{A2})  and  (\hyperlink{(A3)}{A3})  that for any $\varepsilon_1$,  $\varepsilon_2>0$, 
\begin{align*}
		 2\langle&\x(0), A(k)\x(0)+b(\x, k)\rangle_{H}+\| \sigma(\x, k)\|^2_{L_2}\\
   &\leq -2\lambda_1(k)\|\x(0)\|_{H}^{2}+2\langle\x(0)-\mathbf{0}_H, b(\x, k)-b(\mathbf{0},k)+b(\mathbf{0},k)\rangle_{H}\\
   &\quad+\| \sigma(\x, k)-\sigma(\mathbf{0}, k)+\sigma(\mathbf{0}, k)\|^2_{L_2}\\
   &\leq (\alpha(k)-2\lambda_1(k))\|\x(0)\|_{H}^{2}+\beta(k)\int_{-\infty}^{0} \| \x(\theta)\|_{H}^{2} \rho(\d \theta)+\varepsilon_1\|\x(0)\|_H^2+\frac{1}{\varepsilon_1}\|b(\mathbf{0},k)\|_{H}^{2}\\
   &\quad+(1+\varepsilon_2)\| \sigma(\x, k)-\sigma(\mathbf{0}, k)\|^2_{L_2}+\left(1+\frac{1}{\varepsilon_2}\right)\|\sigma(\mathbf{0}, k)\|_{L_2}^2\\
   &\leq \big(\alpha(k)-2\lambda_1(k)+\varepsilon_1+(1+\varepsilon_2)L\big)\|\x(0)\|_{H}^{2}+\big(\beta(k)+(1+\varepsilon_2)L\big)\int_{-\infty}^{0} \| \x(\theta)\|_{H}^{2} \rho(\d \theta)\\
   &\quad+\frac{1}{\varepsilon_1}\|b(\mathbf{0},k)\|_{H}^{2}+\left(1+\frac{1}{\varepsilon_2}\right)\|\sigma(\mathbf{0}, k)\|_{L_2}^2.
   \end{align*}
For any  $\varepsilon>0$, letting $\varepsilon_1=\varepsilon$ and $\varepsilon_2=\varepsilon/L$ implies that there exists a constant $C({\varepsilon})>0$ such that
\begin{equation}\label{Eq:1012:1}
    \begin{aligned} 
		 2\langle&\x(0), A(k)\x(0)+b(\x, k)\rangle_{H}+\| \sigma(\x, k)\|^2_{L_2}\\
   &\leq C({\varepsilon})+(\alpha(k)-2\lambda_1(k)+L+2\varepsilon)\|\x(0)\|_{H}^{2}+(\beta(k) +L+\varepsilon)\int_{-\infty}^{0} \| \x(\theta)\|_{H}^{2} \rho(\d \theta),
\end{aligned}
\end{equation}
where we have used the fact that $N<\infty$. 

Noting that $\mathcal{A}$ is a non-singular $M$-matrix, we have  for any $k\in \S$
\[
(Q\xi)(k)+(\alpha(k)-2\lambda_1(k)+L)\xi(k)=-1
\]
by the definition of $\xi$. Let $\lambda\in (0,2r)$ be a constant  to be determined later. By \eqref{Eq:1012:1} and the generalized \ito~ formula (see, e.g., \cite[Theorem 1.45]{Mao_Yuan_Book_2006}),  we can derive that 
\begin{align*}
		\e ^{\lambda t} &\E\left [\|X(t)\|_H^2\xi(\Lambda(t))\right] \\
 & =\e ^{\lambda s}\E[\|X(s)\|_H^2\xi(\Lambda(s))] +\E \int_{s}^{t}\e ^{\lambda u} \Big[ \lambda\|X(u)\|_H^2\xi(\Lambda(u)) + \|X(u)\|_H^2(Q\xi)(\Lambda(u))\\
 &\quad+\big(2\langle X(u),A(\Lambda(u))X(u)+b(X_u,\Lambda(u))\rangle +\|\sigma(X_u ,\Lambda (u))\|^2_{L_2}\big)\xi(\Lambda(u))\Big]\d u \\
 &\leq \e ^{\lambda s}\E[\|X(s)\|_H^2\xi(\Lambda(s))] +\frac{C(\varepsilon)\xi_{max}}{\lambda}\e^{\lambda t}+\lambda\xi_{max} \E \int_{s}^{t}\e ^{\lambda u}  \|X(u)\|_H^2\d u\\
 &\quad+\E \int_{s}^{t}\e ^{\lambda u}  \big( (Q\xi)(\Lambda(u))+(\alpha(\Lambda(u))-2\lambda_1(\Lambda(u))+L+2\varepsilon)\xi(\Lambda(u))\big)\|X(u)\|_H^2\d u \\
 &\quad+\E \int_{s}^{t}\int_{-\infty}^0\e ^{\lambda u}(\beta(\Lambda(u))+L+\varepsilon)\xi(\Lambda(u))\|X(u+\theta)\|_H^2\rho(\d \theta)\d u\\
 &\leq \xi_{max}\e ^{\lambda s}\|\varphi\|_r^2+\frac{C(\varepsilon)\xi_{max}}{\lambda}\e^{\lambda t}+((\lambda+2\varepsilon)\xi_{max}-1)\E \int_{s}^{t}\e ^{\lambda u}\|X(u)\|_H^2\d u\\
 &\quad+(K+\varepsilon\xi_{max})\E \int_{s}^{t}\int_{-\infty}^0\e ^{\lambda u}\|X(u+\theta)\|_H^2\rho(\d \theta)\d u
\end{align*}
where $K=\max_{1\leq k\leq N}(\beta(k)+L)\xi(k)$. Moreover,  $X_s=\varphi$ implies that $X(t)=\varphi(t-s)$ for any $t\leq s$.  Thus,
    \begin{align}\label{Eq:6}
   \nonumber      \int_{s}^{t}&\int_{-\infty}^0\e ^{\lambda u}\|X(u+\theta)\|_H^2\rho(\d \theta)\d u\\
      \nonumber  &=  \int_{s}^{t}\int_{-\infty}^{s-u}\e ^{\lambda u}\|X(u+\theta)\|_H^2\rho(\d \theta)\d u+\int_{s}^{t}\int_{s-u}^{0}\e ^{\lambda u}\|X(u+\theta)\|_H^2\rho(\d \theta)\d u\\
     \nonumber   &\leq  \int_{s}^{t}\int_{-\infty}^{s-u}\e ^{\lambda u-2r(u-s)}\e^{-2r\theta}\e^{2r(u+\theta-s)}\|\varphi(u+\theta-s)\|_H^2\rho(\d\theta)\d u\\
        &\quad+\int_{s-t}^{0}\int_{s-\theta}^{t}\e ^{\lambda u}\|X(u+\theta)\|_H^2\d u\rho(\d \theta)\\
      \nonumber  &\leq \rho^{(2r)}\|\varphi\|_r^2\e^{2rs}\int_{s}^{t}\e^{(\lambda-2r)u}\d u+\int_{s-t}^{0}\int_{s}^{t+\theta}\e ^{\lambda (v-\theta)}\|X(v)\|_H^2\d v\rho(\d \theta)\\
      \nonumber  &\leq \frac{\rho^{(2r)}}{2r-\lambda}\e^{\lambda s}\|\varphi\|_r^2+\int_{-\infty}^{0}\e ^{-\lambda \theta}\rho(\d \theta)\int_{s}^{t}\e ^{\lambda u}\|X(u)\|_H^2\d u.
    \end{align}
Furhermore,
\begin{equation}\label{Eq:1014:1}
    \begin{aligned}
        \e ^{\lambda t} &\E\left [\|X(t)\|_H^2\xi(\Lambda(t))\right] \leq \widetilde{C}(\varepsilon) \e ^{\lambda s}\|\varphi\|_r^2+\frac{C(\varepsilon)\xi_{max}}{\lambda}\e^{\lambda t}+K_1(\lambda)\E \int_{s}^{t}\e ^{\lambda u}\|X(u)\|_H^2\d u
    \end{aligned}
\end{equation}
where $\widetilde{C}(\varepsilon)=\left(1+\frac{\rho^{(2r)}(\beta_{max}+L+\varepsilon)}{2r-\lambda}\right)\xi_{max}$ and 
\[
K_1(\lambda)=(\lambda+2\varepsilon)\xi_{max}-1+(K+\varepsilon\xi_{max})\int_{-\infty}^{0}\e ^{-\lambda \theta}\rho(\d \theta).
\]
Since $(\beta(k)+L)\xi(k)<1$ for any $1\leq k\leq N$, then $K<1$. Moreover, $K_1(0)=3\varepsilon\xi_{max}-1+K$, so we can choose $\varepsilon>0$ and $\lambda\in (0,2r)$ to be sufficiently small such that $K_1(\lambda)\leq 0$, which, together with \eqref{Eq:1014:1},  implies 
 \begin{equation*}
     \E[\|X(t)\|_H^2]\leq \frac{\E[\|X(t)\|_H^2\xi(\Lambda(t))]}{\xi_{min}}\lesssim 1+\e^{-\lambda(t-s)}\|\varphi\|_r^2,
 \end{equation*}
for any $t\geq s$. The proof is complete.
\end{proof}
\begin{lemma}\label{Lem:1013:1}
   Let assumptions of Lemma \ref{lem:yizhiyoujie} be satisfied. There exists a positive constant $\hat{\lambda}$ such that  for  any initial data  $(\varphi,i) ,(\psi,i)\in \mathscr{C}_{r}\times \S$  and   $-\infty<s\leq t <\infty,$ 
    \begin{equation}\label{EQ:GAMMA:01}
    \mathbb{E}[\left\|X(t;s,(\varphi,i))-X(t; s,(\psi,i))\right\|_{H}^{2}] \lesssim\|\varphi-\psi\|_{r}^{2} \e^{-\hat{\lambda}\left(t-s\right)}.
\end{equation}

\end{lemma}
\begin{proof}
Set
\begin{equation}\label{EQ:1013:01}
    \Gamma(t):=X(t;s,(\varphi,i))-X(t;s,(\psi,i)).
\end{equation}
Then, it follows from  \eqref{EQ:FRSPDE:02} that 
\[
\begin{aligned}
\d \Gamma(t)= & A(\Lambda(t ; s, i)) \Gamma(t) \d t+\big(b\left(X_{t}(s,(\varphi, i)), \Lambda(t ; s, i)\right)-b\left(X_{t}(s,(\psi, i)),\Lambda(t ; s, i)) \right)\big)\d t\\
& +\big(\sigma\left(X_{t}(s,(\varphi, i)), \Lambda(t ; s, i)\right)-\sigma\left(X_{t}(s,(\psi, i)), \Lambda(t ; s, i)\right)\big) \d \overline{W}(t).
\end{aligned}
\]

For the sake of notational simplicity,  write $\Lambda(t ; s, i)=\Lambda(t)$. 
Let $\hat{\lambda}\in (0,2r)$ be a constant  to be determined later. Similar to \eqref{Eq:6}, we have
\begin{equation}
    \begin{aligned}
         \int_{s}^{t}&\int_{-\infty}^0\e ^{\hat{\lambda} u}\|\Gamma(u+\theta)\|_H^2\rho(\d \theta)\d u\\
        &=  \int_{s}^{t}\int_{-\infty}^{s-u}\e ^{\hat{\lambda} u}\|\Gamma(u+\theta)\|_H^2\rho(\d \theta)\d u+\int_{s}^{t}\int_{s-u}^{0}\e ^{\hat{\lambda} u}\|\Gamma(u+\theta)\|_H^2\rho(\d \theta)\d u\\
        &\leq \rho^{(2r)}\|\varphi-\psi\|_r^2\e^{2rs}\int_{s}^{t}\e^{(\hat{\lambda}-2r)u}\d u+\int_{s-t}^{0}\int_{s}^{t+\theta}\e ^{\hat{\lambda} (v-\theta)}\|\Gamma(v)\|_H^2\d v\rho(\d \theta)\\
        &\leq \frac{\rho^{(2r)}}{2r-\hat{\lambda}}\e^{\hat{\lambda} s}\|\varphi-\psi\|_r^2+\int_{-\infty}^{0}\e ^{-\hat{\lambda} \theta}\rho(\d \theta)\int_{s}^{t}\e ^{\hat{\lambda} u}\|\Gamma(u)\|_H^2\d u.
    \end{aligned}
\end{equation}
Lemma \ref{lem:yizhiyoujie} implies that $\E[\|\Gamma(t)\|_H^2]<\infty$ for any $t\geq 0$. Furthermore, by the generalized \ito~  formula, we have 
    \begin{align*}
        \e^{\hat{\lambda} t}&\E [\|\Gamma(t)\|_H^2\xi(\Lambda(t))] \\
        & = \e^{\hat{\lambda} s}\E[\|\Gamma(s)\|_H^2\xi(\Lambda(s))] +\E\int_{s}^t\e^{\hat{\lambda}u}\bigg[ \hat{\lambda}\|\Gamma(u)\|_H^2\xi(\Lambda(u))+\|\Gamma(u)\|_H^2(Q\xi)(\Lambda(u))\\
        & \quad +\Big( 2\langle \Gamma(u),A(\Lambda(u))\Gamma(u)+b(X_u(s,(\varphi,i)),\Lambda(u)) -b(X_u(s,(\psi,i)),\Lambda(u))\rangle \\
        &\quad +\|\sigma(X_u(s,(\varphi,i)),\Lambda(u))-\sigma(X_u(s,(\psi,i)),\Lambda(u))\|_{L_2}^2\Big) \xi(\Lambda(u)) \bigg] \d u\\
        & \leq \xi_{max}\e^{\hat{\lambda} s} \|\varphi-\psi\|_r^2 +\hat{\lambda}\xi_{max} \E \int_{s}^{t}\e ^{\hat{\lambda} u}  \|\Gamma(u)\|_H^2\d u\\
        &\quad+\E\int_{s}^t\e^{\hat{\lambda} u}\Big((Q\xi)(\Lambda(u))+\big(\alpha(\Lambda(u))-2\lambda_1(\Lambda(u))+L\big)\xi(\Lambda(u))\Big)\|\Gamma(u)\|_H^2\d u\\
        &\quad +K\E\int_{s}^t\int_{-\infty}^0\e^{\hat{\lambda} u}\|\Gamma(u+\theta)\|_H^2\rho(\d \theta)\d u\\
        & \leq \left(1+\frac{\rho^{(2r)}(\beta_{max}+L)}{2r-\hat{\lambda}}\right)\xi_{max}\e^{\hat{\lambda} s}\|\varphi-\psi\|_r^2+K_2(\hat{\lambda})\E\int_{s}^{t}\e ^{\hat{\lambda} u}\|\Gamma(u)\|_H^2\d u,
    \end{align*}
where$$K_2(\hat{\lambda}):=\hat{\lambda} \xi_{max}-1+K\int_{-\infty}^{0}\e^{\hat{\lambda}\theta}\rho(\d\theta).$$   Since $K<1$, then $K_2(0) =-1 +K<0$. So we can choose $\hat{\lambda}\in(0,2r) $ sufficiently small such that $K_2(\hat{\lambda})\leq 0.$ Thus,
\begin{equation*}
           \e^{\hat{\lambda} t}\E [\|\Gamma(t)\|_H^2\xi(\Lambda(t))] \lesssim \e^{\hat{\lambda} s}\|\varphi-\psi\|_r^2,       
\end{equation*}
which implies that
\[
\E[\|\Gamma(t)\|_H^2]\leq \frac{\E[\|\Gamma(t)\|_H^2\xi(\Lambda(t))]}{\xi_{min}}\lesssim \|\varphi-\psi\|_r^2\e^{-\hat{\lambda}(t-s)}.
\]
The proof is complete.
\end{proof}
\subsection{Boundedness and Convergence of $X_t$}\label{subsect:4.2}
\begin{theorem}
\label{Thm:4.3}
     Let assumptions of Lemma \ref{lem:yizhiyoujie} be satisfied.  For any initial data  $(\varphi,i)\in \mathscr{C}_{r}\times \S$  and $-\infty<s\leq t <\infty,$ we have
    \begin{equation}\label{EQ:1013:02}
    \mathbb{E}[\left\|X_{t}(s,(\varphi,i))\right\|_{r}^{2}] \lesssim 1+\e^{-\lambda(t-s)}\|\varphi\|_{r}^{2},
\end{equation}
where $\lambda$ is given in \eqref{Eq:1015:1}.
\end{theorem}
\begin{proof}
For the sake of convenience in notation, write $X_{t}(s,(\varphi,i))=X_{t}$ in this proof. By the definition of $\|\cdot\|_r$, we have
    \begin{equation}\label{Eq:1013:1}
        \begin{aligned}
\mathbb{E}[\left\|X_{t}\right\|_{r}^{2}]&=\mathbb{E}\left(\sup _{-\infty<\theta \leq 0} \e^{2 r \theta}\|X(t+\theta)\|_{H}^{2}\right)\\
&\leq \mathbb{E}\left(\sup _{-\infty<\theta \leq s-t} \e^{2 r \theta}\|X(t+\theta)\|_{H}^{2}\right)+\mathbb{E}\left(\sup _{s-t<\theta \leq 0} \e^{2 r \theta}\|X(t+\theta)\|_{H}^{2}\right)\\
&=\mathbb{E}\left(\sup _{-\infty<\theta \leq s-t} \e^{2 r \theta}\|\varphi(t+\theta-s)\|_{H}^{2}\right)
+\mathbb{E}\left(\sup_{s<u \leqslant t} \e^{2 r(u-t)}\|X(u)\|_{H}^{2}\right)\\
&=\mathbb{E}\left(\sup_{-\infty< \theta^{\prime} \leqslant 0} \e^{2r\left(\theta^{\prime}-t+s\right)}\left\|\varphi\left(\theta^{\prime}\right)\right\|_{H}^{2}\right)
+\e^{-2rt} \mathbb{E}\left(\sup _{s<u\leq t } \e^{2 r u}\|X(u)\|_{H}^{2}\right)\\
&=\e^{-2r(t-s)}\|\varphi\|_{r}^{2}+\e^{-2rt}\mathbb{E}\left(\sup _{s<u \leq t} \e^{2 r u}\|X(u)\|_{H}^{2}\right).
\end{aligned}
    \end{equation}
Then,    we shall evaluate the term $\mathbb{E}\left(\sup _{s<u \leq t} \e^{2 r u}\|X(u)\|_{H}^{2}\right)$. Applying \ito's formula to $\e^{2rt}\|X(t)\|_H^2,$ it follows that
\begin{align*}
        \e^{2rt}\|X(t)\|_H^2&=\e^{2 r s}\|X(s)\|_{H}^{2}+\int_{s}^{t} \e^{2 r u}\left(2 r\|X(u)\|_{H}^{2}+2\langle  X(u), A(\Lambda(u)) X(u)+b(X_u, \Lambda(u))\rangle\right. \\
& \quad+\left\|\sigma\left(X_{u}, \Lambda(u)\right)\right\|_{L_2}^{2}) \d u +2\int_{s}^{t} \e^{2 r u}\langle  X(u), \sigma(X_u, \Lambda(u)) \d \overline{W}(u)\rangle\\
        &\leq \e^{2rs}\|\varphi\|_r^2+\int_{s}^{t}\e^{2ru}\big[C(\varepsilon)+(2r+\alpha(\Lambda(u))-2\lambda_1(\Lambda(u))+L+2\varepsilon)\|X(u)\|_H^2\big]\d u\\
        &\quad +\int_{s}^{t}\e^{2ru}(\beta(\Lambda(u))+L+\varepsilon)\int_{-\infty}^{0}\|X(u+\theta)\|_H^2\rho(\d \theta)\d u\\
        &\quad+2\int_{s}^{t}\e^{2ru}\langle X(u),\sigma(X_u,\Lambda(u))\d \overline{W}(u)\rangle\\
        &\leq \e^{2rs}\|\varphi\|_r^2+\frac{C(\varepsilon)}{2r}\e^{2rt}+(2r+\alpha_{max}\vee 0+L+2\varepsilon)\int_{s}^{t}\e^{2ru}\|X(u)\|_H^2\d u\\
        &\quad +(\beta_{max}+L+\varepsilon)\int_{s}^{t}\int_{-\infty}^{0}\e^{2ru}\|X(u+\theta)\|_H^2\rho(\d \theta)\d u\\
        &\quad+2\int_{s}^{t}\e^{2ru}\langle X(u),\sigma(X_u,\Lambda(u))\d \overline{W}(u)\rangle,
    \end{align*}

where we have used \eqref{Eq:1012:1} in the first inequality.

Define a family of non-decreasing stopping times $\{\tau_n^{\prime}\}_{n\geq 1}$  by
\begin{equation*}
        \tau_{n}^{\prime}:=\inf\{t\geq 0:\|X(t)\|_H\geq n\}\quad \text { for any } n \geq 1.
\end{equation*}
We claim that $\tau_n^{\prime}$ tends to $\infty$ almost surely, as $n\to\infty$. In fact, for any given $T>0$, we have $\|X(\tau_n^{\prime}\wedge T)\|_H^2 \geq n^2$ on the set $\left\{\tau_n^{\prime} \leq T\right\}$. Therefore, it follows from Markov's inequality and Lemma \ref{lem:yizhiyoujie} that
\begin{equation}\label{EQ:1030:11}
    		\mathbb{P}\left\{\tau_n^{\prime} \leq T\right\}\leq \mathbb{P}\{\|X(\tau_n^{\prime}\wedge T)\|_H^2 \geq n^2\}\leq \frac{\mathbb{E}\|X(\tau_n^{\prime}\wedge T)\|_H^2 }{n^2}\lesssim \frac{1+\|\varphi\|_r^2}{n^2}\rightarrow 0, 
\end{equation}
as  $n \rightarrow \infty$. Similar to \eqref{Eq:1012:1}, we can deduce that
\begin{equation}\label{EQ:1020:09}
    \|\sigma({\bf x}, k)\|_{L_2}^{2} 
 \leq C(\varepsilon)+(L+\varepsilon)\|\x(0)\|_H^2+(L+\varepsilon) \int_{-\infty}^{0}\|{\bf x}(\theta)\|_{H}^{2} \rho(\d \theta).
\end{equation}
Thus, \eqref{EQ:1020:09}, BDG's inequality and Young's inequality, imply
\begin{equation}\label{Eq:1230}
    \begin{aligned}
           2\E&\left(\sup_{s<u\leq t\wedge\tau_{n}^{\prime}}\int_{s}^{u}\e^{2rv}\langle X(v),\sigma(X_v,\Lambda(v))\d \overline{W}(v)\rangle_H\right)\\
        &\leq 8\sqrt{2}\E\left(\int_{s}^{t\wedge\tau_{n}^{\prime}}\e^{4ru}\|X(u)\|_H^2\|\sigma(X_u,\Lambda(u))\|_{L_2}^2\d u\right)^{1/2}\\
        &\leq \frac{1}{2}\E\left(\sup_{s<u\leq t\wedge\tau_{n}^{\prime}}\e^{2ru}\|X(u)\|_H^2\right)+64\E\int_{s}^{t\wedge\tau_{n}^{\prime}}\e^{2ru}\|\sigma(X_u,\Lambda(u))\|_{L_2}^2\d u\\
      &  \leq \frac{1}{2}\E\left(\sup_{s<u\leq t\wedge\tau_{n}^{\prime}}\e^{2ru}\|X(u)\|_H^2\right)+\frac{32C(\varepsilon)}{r}\e^{2rt}+64(L+\varepsilon)\E\int_{s}^{t\wedge\tau_{n}^{\prime}}\e^{2ru}\|X(u)\|_H^2\d u\\
      &\quad+64(L+\varepsilon) \E\int_{s}^{t\wedge\tau_{n}^{\prime}}\int_{-\infty}^{0}\e^{2ru}\|X(u+\theta)\|_H^2\rho(\d \theta)\d u.
    \end{aligned}
\end{equation}
 Applying a similar argument as in \eqref{Eq:6}, we obtain
\begin{equation*}\label{EQ:1016:01}
\int_{s}^{t\wedge\tau_{n}^{\prime}}\int_{-\infty}^{0}\e^{2ru}\|X(u+\theta)\|_H^2\rho(\d \theta)\d u \leq  \rho^{(2r)}\e^{2rs}(t-s)\|\varphi\|_r^2+\rho^{(2r)}\int_{s}^{t\wedge\tau_{n}^{\prime}}\e^{2r u} \|X(u)\|_H^2\d u.
\end{equation*}
In summary, the calculations presented above yield
    \begin{align*}
        &\E\left(\sup_{s<u\leq t\wedge\tau_{n}^{\prime}}\e^{2ru}\|X(u)\|_H^2\right)\\
        &\quad \leq 2\e^{2rs}\|\varphi\|_r^2+\frac{65C(\varepsilon)}{r}\e^{2rt}+2(2r+\alpha_{max}\vee 0+65L+66\varepsilon)\E\int_{s}^{t\wedge\tau_{n}^{\prime}}\e^{2ru}\|X(u)\|_H^2\d u\\
        &\quad \quad+2(\beta_{max}+65L+65\varepsilon)\E\int_{s}^{t\wedge\tau_{n}^{\prime}}\int_{-\infty}^{0}\e^{2ru}\|X(u+\theta)\|_H^2\rho(\d \theta)\d u\\
        &\quad \lesssim  \e^{2rt}+(1+(t-s)) \e^{2rs}\|\varphi\|_r^2+\E\int_{s}^{t\wedge\tau_{n}^{\prime}}\e^{2ru}\|X(u)\|_H^2\d u.
    \end{align*}
Letting $n\to \infty$, the monotone convergence theorem implies 
\begin{equation}\label{Eq:1017:1}
        \begin{aligned}
        &\E\left(\sup_{s<u\leq t}\e^{2ru}\|X(u)\|_H^2\right) \lesssim  \e^{2rt}+(1+(t-s)) \e^{2rs}\|\varphi\|_r^2+\E\int_{s}^{t}\e^{2ru}\|X(u)\|_H^2\d u.
    \end{aligned}
    \end{equation}
    By inserting \eqref{Eq:1017:1} into \eqref{Eq:1013:1} and applying Lemma \ref{lem:yizhiyoujie}, we have
%Inserting \eqref{Eq:1017:1} in \eqref{Eq:1013:1}  and using Lemma \ref{lem:yizhiyoujie}, we have
\[
\begin{aligned}
    \mathbb{E}[\left\|X_{t}\right\|_{r}^{2}]
    &\lesssim 1+\e^{-2r(t-s)}(2+(t-s)) \|\varphi\|_r^2+\e^{-2rt}\int_{s}^{t}\e^{2ru}\E\|X(u)\|_H^2\d u\\
    &\lesssim 1+\Big(\e^{-(2r-\lambda)(t-s)}(2+(t-s))\Big)\e^{-\lambda(t-s)}\|\varphi\|_r^2+\e^{-\lambda(t-s)}\|\varphi\|_r^2\\
&\lesssim 1+\e^{-\lambda(t-s)}\|\varphi\|_r^2,
\end{aligned}
\]
where we have used the fact that for any $c>0$ and $t>0$, $t\e^{-ct}\leq \frac{1}{c\e}$.  The proof is complete.
\end{proof}
%To proceed, for the difference of two segment processes started from two different initial states, we have the following estimation.
%To proceed, we provide the following estimation for the difference between two segment processes initiated from different initial states.
\begin{theorem}
\label{Thm:4.4}
     Let assumptions of Lemma \ref{lem:yizhiyoujie} be satisfied.  For  any initial data  $(\varphi,i) ,(\psi,i)\in \mathscr{C}_{r}\times \S$  and   $-\infty<s\leq t <\infty,$ we have
    \begin{equation}\label{Eq:1018:1}
    \mathbb{E}[\left\|X_{t}(s,(\varphi,i))-X_{t}(s,(\psi,i))\right\|_{r}^{2}] \lesssim\|\varphi-\psi\|_{r}^{2} \e^{-\hat{\lambda}\left(t-s\right)},
\end{equation}
where $\hat{\lambda}$ is given in \eqref{EQ:GAMMA:01}.
\end{theorem}
\begin{proof}
    Similar to \eqref{Eq:1013:1}, we have
    \begin{equation}\label{Eq:1013:4}   \mathbb{E}[\left\|\Gamma_{t}\right\|_{r}^{2}]\leq \e^{-2 r (t-s)}\|\varphi-\psi\|_{r}^{2}+\e^{-2rt}\mathbb{E}\left(\sup _{s<u \leq t} \e^{2 r u}\|\Gamma(u)\|_{H}^{2}\right),
    \end{equation}
    where $\Gamma_t$ is the segment process of $\Gamma(t)$ defined by \eqref{EQ:1013:01}.    By carrying out an argument similar to  \eqref{Eq:6},  it follows that
\begin{equation}\label{Eq:1013:2}
         \int_{s}^{t} \int_{-\infty}^{0} \e^{2 r u} \| \Gamma( u+\theta) \|_{H}^{2} \rho(\d \theta) \d u\leq \rho^{(2r)}(t-s)\e^{2r s}\|\varphi-\psi\|_r^2+\rho^{(2r)}\int_{s}^{t} \e^{2 r u}\|\Gamma(u)\|_{H}^{2} \d u.
\end{equation} 
    Applying \ito's formula to $\e^{2rt}\|\Gamma(t)\|_H^2$, we have
        \begin{align*}
\e^{2rt}\|\Gamma(t)\|_H^2&\leqslant \e^{2 r s}\|\varphi-\psi\|_{r}^{2}+2 r \int_{s}^{t} \e^{2 r u}\|\Gamma(u)\|_{r}^{2} \d u -2 \int_{s}^{t} \e^{2 r u} \lambda_{1}(\Lambda(u))\|\Gamma(u)\|_{H}^{2} \d u\\
&\quad+\int_{s}^{t} \e^{2 r u}\left((\alpha(\Lambda(u))+L)\|\Gamma(u)\|_{H}^{2}+(\beta(\Lambda(u))+L) \int_{-\infty}^{0}\left\|\Gamma(u+\theta)\right\|_{H }^{2}\rho(\d \theta)\right)\d u \\
&\quad+2 \int_{s}^{t} \e^{2 r u}\left\langle\Gamma(u),\big(\sigma(X_{u}(s, (\varphi, i)), \Lambda(u))-\sigma\left(X_{u}(s, (\psi, i)),\Lambda(u)\right)\big) \d \overline{W}(u)  \right\rangle_H\\
&\leqslant \e^{2 r s}\|\varphi-\psi\|_{r}^{2}+(2 r+\alpha_{max}\vee 0+L) \int_{s}^{t} \e^{2 r u}\|\Gamma(u)\|_{H}^{2} \d u \\
&\quad+\left(\beta_{max }+L\right) \int_{s}^{t} \int_{-\infty}^{0} \e^{2 r u} \| \Gamma( u+\theta) \|_{H}^{2} \rho(\d \theta) \d u \\
&\quad+2 \int_{s}^{t} \e^{2 r u}\left\langle\Gamma(u),\big(\sigma(X_{u}(s, (\varphi, i)), \Lambda(u))-\sigma\left(X_{u}(s, (\psi, i)),\Lambda(u)\right)\big) \d \overline{W}(u)  \right\rangle_H\\
&\leq C_1(1+(t-s))\e^{2 r s}\|\varphi-\psi\|_{r}^{2}+C_2\int_{s}^{t} \e^{2 r u}\|\Gamma(u)\|_{H}^{2} \d u \\
&\quad+2 \int_{s}^{t} \e^{2 r u}\left\langle\Gamma(u),\big(\sigma(X_{u}(s, (\varphi, i)), \Lambda(u))-\sigma\left(X_{u}(s, (\psi, i)),\Lambda(u)\right)\big) \d \overline{W}(u)  \right\rangle_H,
        \end{align*}
    where $C_1=1\vee \left(\beta_{max }+L\right)\rho^{(2r)}$, $C_2=2 r+\alpha_{max}\vee 0+L+ \left(\beta_{max }+L\right)\rho^{(2r)}$ and we have used \eqref{Eq:1013:2} in the last inequality. 

    Define a family of non-decreasing stopping times $\{\hat{\tau}_n\}_{n\geq 1}$  by
\begin{equation*}
        \hat{\tau}_n:=\inf\{t\geq 0:\|X_{t}(s,(\varphi,i))\|_H\vee \|X_{t}(s,(\psi,i))\|_H\geq n\}\quad \text { for any } n \geq 1.
\end{equation*}
Similar to \eqref{EQ:1030:11}, $\hat{\tau}_n$ tends to $\infty$ almost surely, as $n\to\infty$.  Moreover,  as in \eqref{Eq:1230}, 
    \begin{align*}
2\mathbb{E}&\left(\sup_ { s < u \leq t\wedge\hat{\tau}_n } \int _ { s } ^ { u } \e ^ { 2 r v } \left\langle\Gamma(v),\big(\sigma(X_{v}(s,(\varphi, i )), \Lambda(v))-\sigma\left(X_{v}(s,(\psi, i)), \Lambda(v)\right) \big)\d\overline{W}(v)\right\rangle_{H}\right)\\
&\leq \frac{1}{2} \mathbb{E}\left(\sup _{s<u \leq t\wedge\hat{\tau}_n } \e^{2 r u} \| \Gamma(u) \|_{H}^{2}\right)+64 L \mathbb{E}\int_{s}^{t\wedge\hat{\tau}_n } \e^{2 r u}\|\Gamma(u)\|_{H}^{2} \d u\\
&\quad+64 L \mathbb{E}\int_{s}^{t\wedge\hat{\tau}_n } \int_{-\infty}^{0} \e^{2 r u}\|\Gamma(u+\theta)\|_{H}^{2} \rho( \d \theta)\d u\\
&\leq \frac{1}{2} \mathbb{E}\left(\sup _{s<u \leq t\wedge\hat{\tau}_n } \e^{2 r u} \| \Gamma(u) \|_{H}^{2}\right)+64 L \rho^{(2 r)}(t-s)\e^{2 r s} \|\varphi-\psi\|_{r}^{2} \\
&\quad+64 L\left(1+\rho^{(2 r)}\right) \mathbb{E}\int_{s}^{t\wedge\hat{\tau}_n } \e^{2 r u}\left\|\Gamma(u)\right\|_{H}^{2} \d u,
        \end{align*}
where we have used  (\hyperlink{A3}{A3}). Combining the above calculations with implies 
   \begin{equation*}
           \mathbb{E}\left(\sup _{s<u \leq t\wedge\hat{\tau}_n } \e^{2 r u}\|\Gamma(u)\|_{H}^{2}\right)\lesssim  (1+(t-s))\e^{2 r s}\|\varphi-\psi\|_{r}^{2}+\E\int_{s}^{t\wedge\hat{\tau}_n } \e^{2 r u}\|\Gamma(u)\|_{H}^{2} \d u.
   \end{equation*}
  Letting $n\to\infty$ and using Lemma \ref{Lem:1013:1} gives 
  \begin{equation}\label{Eq:1013:3}
       \begin{aligned}
           \mathbb{E}\left(\sup _{s<u \leq t } \e^{2 r u}\|\Gamma(u)\|_{H}^{2}\right)&\lesssim(1+(t-s))\e^{2 r s}\|\varphi-\psi\|_{r}^{2}+\E\int_{s}^{t } \e^{2 r u}\|\Gamma(u)\|_{H}^{2} \d u\\
           &\lesssim \big((1+(t-s))\e^{2rs}+\e^{ -\hat{\lambda}(t-s)+2rt})\big)\|\varphi-\psi\|_{r}^{2}.
       \end{aligned}
   \end{equation}
  Plugging \eqref{Eq:1013:3} into \eqref{Eq:1013:4} yields 
\begin{equation*}
\begin{aligned}
     \mathbb{E}[\left\|\Gamma_{t}\right\|_{r}^{2}]&\lesssim \e^{-2 r (t-s)}\|\varphi-\psi\|_{r}^{2}+\e^{-2rt} \big((1+(t-s))\e^{2rs}+\e^{-\hat{\lambda}(t-s)+2rt}))\|\varphi-\psi\|_{r}^{2}\\
    &\lesssim \big(\e^{-\hat{\lambda} (t-s)}+(t-s)\e^{-(2r-\hat{\lambda})(t-s)}\e^{-\hat{\lambda} (t-s)}\big)\|\varphi-\psi\|_{r}^{2}\\
    &\lesssim \e^{-\hat{\lambda} (t-s)}\|\varphi-\psi\|_{r}^{2},
\end{aligned}
\end{equation*}
  where we have used the fact that  $\hat{\lambda}\in (0,2r)$.    The proof is complete.
    \end{proof}
\subsection{Exponential Erogodicity}\label{subsect:4.3}

Now we give the main theorem in this section.
\begin{theorem}\label{THM:MAIN:01}
    Let assumptions of Lemma \ref{lem:yizhiyoujie} be satisfied.  Then, the FSPDEwM \eqref{EQ:FRSPDE:01} has a unique invariant measure \( \mu \in \mathscr{P}(\mathscr{C}_{r} \times \mathbb{S}) \), which is also exponentially mixing.
\end{theorem}

\begin{proof}
    %The proof of this theorem is divided into the following three steps. 
The proof  is divided into three  steps as follows.

\noindent Step 1: Existence of an invariant measure. Let 
\begin{equation*}
    L^1(\Omega;\mathscr{C}_r\times
     \S):=\Big\{(Y_1,Y_2):(\Omega,\mathscr{F},\P)\to (\mathscr{C}_r\times
     \S,\mathcal{B}(\mathscr{C}_r\times
     \S)) \mid \E[ d((Y_1,Y_2),(\textbf{0},k_0))]<\infty\Big\},
\end{equation*}
where $k_0 \in \S $ is arbitrarily given. We can show that $L^1(\Omega;\mathscr{C}_r\times
     \S)$ is complete, see Remark \ref{RMK:05} for more details.
For any fixed $t\in \R$,  we want to prove that the sequence of random variables $\{(X_{t}(s, (\varphi,i)),\Lambda(t;s,i))\}_{s\leq t}$ satisfies the Cauchy condition in $L^1(\Omega;\mathscr{C}_r\times\mathbb{S})$ as $s\to -\infty$.  For this aim, let us define the product probability space
\[
(\widetilde{\Omega},\widetilde{\mathscr{H}},\widetilde{\mathbb{P}})=(\Omega\times\Omega,\mathscr{H}\otimes\mathscr{H},\mathbb{P}\times \mathbb{P}).
\]
For any fixed $t\in \R$, we consider the process $(X_{t}(s_1, (\varphi,i)),\Lambda(t;s_1,i), X_{t}(s_2, (\varphi,i)),\Lambda(t;s_2,i))$ on this  product probability space for any initial data  $(\varphi,i) ,(\psi,i)\in \mathscr{C}_{r}\times \S$  and initial times   $-\infty<s_1 \leq s_2 \leq t <\infty$.  Note that the Markov chains  $(\Lambda(t; s_1, i))_{t\geq s_1 } $ and  $(\Lambda(t; s_2, i))_{t\geq s_2}$ are two different  process. Let us define the following stopping time
\begin{equation}\label{Eq:01041}
    \tau:=\inf \left\{t \geq s_2: \Lambda(t;s_1,i)=\Lambda(t;s_2,i)\right\} .
\end{equation}
Since \( \mathbb{S} \) is a finite set and \( Q \) is irreducible, it is well known that there exists a constant \( \theta>0 \) such that
\begin{equation}\label{Eq:1018:2}
    \widetilde{\mathbb{P}}(\tau>t) \leq \mathrm{e}^{-\theta(t-s_2)}, \quad t\geq s _2 .
\end{equation}
From the definition of distance \( d(\cdot,\cdot) \), we have
\begin{equation}\label{eq:xiaolong}
\begin{aligned}
\widetilde{\mathbb{E}}&\Big[  d\Big(\big(X_{t}(s_1,(\varphi, i)), \Lambda(t;s_1,i)\big),\big(X_{t}(s_2,(\varphi, i)), \Lambda(t;s_2,i)\big)\Big)\Big] \\
&=  \widetilde{\mathbb{E}}\left[\left\|X_{t}(s_1,(\varphi, i))-X_{t}(s_2,(\varphi, i))\right\|_{r}+\ell\left(\Lambda(t;s_1,i), \Lambda(t;s_2,i)\right)\right] \\
&=  \widetilde{\mathbb{E}}\left[\left(\|X_{t}(s_1,(\varphi, i))-X_{t}(s_2,(\varphi, i))\|_{r}+\ell\left(\Lambda(t;s_1,i), \Lambda(t;s_2,i)\right)\right)\1_{\{\tau \leqslant (t+s_2) / 2\}}\right] \\
&\quad +\widetilde{\mathbb{E}}\left[(\left\|X_{t}(s_1,(\varphi, i))-X_{t}(s_2,(\varphi, i))\right\|_{r}+\ell\left(\Lambda(t;s_1,i), \Lambda(t;s_2,i)\right))\1_{\{\tau>(t+s_2) / 2\}}\right] \\
&=:J_1(t)+J_2(t),
\end{aligned}
\end{equation}
where $\widetilde{\E}$ denotes the expectation with respect to $\widetilde{\P}$.

Recall the process $(X_{t}(s, (\varphi,i)),\Lambda(t;s,i))_{t\geq s}$ is a strong Markovian process. Hence, the process $(X_{t}(s_1, (\varphi,i)),\Lambda(t;s_1,i), X_{t}(s_2, (\varphi,i)),\Lambda(t;s_2,i))_{t\geq s_2}$ admits the strong Markovian property with respect to the natural filtration $\widetilde{\mathscr{H}}_t.$ Then, it follows from Theorems \ref{Thm:4.3} and \ref{Thm:4.4} that 
\begin{align}\label{eq:dianying}
\nonumber&\quad J_1(t)\\
\nonumber&=\widetilde{\mathbb{E}}\left[\widetilde{\mathbb{E}}\left[\left(\|X_{t}(s_1,(\varphi, i))-X_{t}(s_2,(\varphi, i))\|_{r}+\ell\left(\Lambda(t;s_1,i), \Lambda(t;s_2,i)\right)\right)\1_{\{\tau \leq (t+s_2) / 2\}} \mid \widetilde{\mathscr{H}}_{\tau}\right]\right] \\
\nonumber& =\widetilde{\mathbb{E}}\left[ \widetilde{\mathbb{E}}\left[\left\|X_{t}(\tau,(X_{\tau}(s_1,(\varphi, i)), \Lambda(\tau;s_1,i))-X_{t}(\tau,(X_{\tau}(s_2,(\varphi, i)), \Lambda(\tau;s_2,i))\right\|_{r}\right]\1_{\{\tau \leq (t+s_2) / 2\}}\right] \\
& \lesssim \widetilde{\mathbb{E}}\left(\e^{-\hat{\lambda}(t-\tau)/2} \left\|X_{\tau}(s_1,(\varphi, i))-X_{\tau}(s_2,(\varphi, i))\right\|_{r}\1_{\{\tau \leq (t+s_2) / 2\}}\right)\\
\nonumber& \lesssim \widetilde{\mathbb{E}}\left(\left\|X_{\tau}(s_1,(\varphi, i))\right\|_{r}+\left\|X_{\tau}(s_2,(\varphi, i))\right\|_{r}\right) \e^{-\hat{\lambda} (t-s_2) / 4} \\
\nonumber& \lesssim\left(1+\|\varphi\|_{r}\right) \e^{-\hat{\lambda} (t-s_2) / 4}.
\end{align}
%On the other hand, by \hold's inequality, it follows from \eqref{Eq:1018:2} and Theorem \ref{Thm:4.3} that
Meanwhile, applying Hölder's inequality along with \eqref{Eq:1018:2} and Theorem \ref{Thm:4.3} yields that 
\begin{align}\label{eq:cifff}
\nonumber J_2(t)&\leq \Big[\widetilde{\mathbb{E}}\big(\left\|X_{t}(s_1,(\varphi, i))-X_{t}(s_2,(\varphi, i))\right\|_{r}+\ell(\Lambda(t;s_1,i), \Lambda(t;s_2,i))\big)^{2}\Big]^{1 / 2}\\
\nonumber &\qquad\times [\widetilde{\mathbb{P}}(\tau>(t+s_2) / 2)]^{1 / 2}  \\
&\lesssim \left(1+ \widetilde{\mathbb{E}}\left\|X_{t}(s_1,(\varphi, i))-X_{t}(s_2,(\varphi, i))\right\|_{r}^{2}\right)^{1 / 2}\mathrm{e}^{-\theta (t-s_2) / 4} \\
\nonumber &\lesssim \left(1+ \widetilde{\mathbb{E}}\left\|X_{t}(s_1,(\varphi, i))\right\|_{r}^{2}+ \widetilde{\mathbb{E}}\left\|X_{t}(s_2,(\varphi, i))\right\|_{r}^{2}\right)^{1 / 2} \mathrm{e}^{-\theta (t-s_2) / 4}\\
\nonumber &\lesssim \left(1+\|\varphi\|_{r}\right) \mathrm{e}^{-\theta (t-s_2) / 4} .
\end{align}

Inserting \eqref{eq:dianying}  and \eqref{eq:cifff}  into \eqref{eq:xiaolong}, we obtain
\begin{equation}\label{Eq:3}
    \widetilde{\mathbb{E}}\Big[  d\Big(\big(X_{t}(s_1,(\varphi, i)), \Lambda(t;s_1,i)\big),\big(X_{t}(s_2,(\varphi, i)), \Lambda(t;s_2,i)\big)\Big)\Big]\lesssim \left(1+\|\varphi\|_{r}\right) \mathrm{e}^{-\kappa (t-s_2)} .
\end{equation}
where  \( \kappa=\hat{\lambda} / 4\wedge \theta / 4\).  Thus, there exists a random variable  $\eta_t(\varphi,i)\in L^1(\Omega;\mathscr{C}_r\times\mathbb{S}) $ such that
\begin{equation}\label{Eq:1027:1}
    \lim_{s\to-\infty}\widetilde{\mathbb{E}}\Big[  d\Big(\big(X_{t}(s,(\varphi, i)), \Lambda(t;s,i)\big),\eta_t(\varphi,i)\Big)\Big]=0.
\end{equation}

Next, following the argument to derive \eqref{Eq:3}, we can similarly obtain that
\begin{equation}\label{Eq:4}
    \widetilde{\mathbb{E}}\Big[  d\Big((X_{t}(s,(\varphi, i)), \Lambda(t;s,i)),(X_{t}(s,(\psi, j)), \Lambda(t;s,j))\Big)\Big]\lesssim (1+\|\varphi\|_r+\|\psi\|_r)\mathrm{e}^{-\kappa (t-s)} .
\end{equation}
Then,  \( \eta_{t}(\varphi,i) \) is independent of the initial value \( (\varphi,i)\in \mathscr{C}_r\times\mathbb{S}\), which is simply denoted by \( \eta_{t} \). Indeed, by \eqref{Eq:3} and \eqref{Eq:4}, we have 
\[
\begin{aligned}
\widetilde{\mathbb{E}}&[  d(\eta_t(\varphi,i),\eta_t(\psi,j))] \\
&\leq  \widetilde{\mathbb{E}}\Big[  d\Big((X_{t}(s,(\varphi, i)), \Lambda(t;s,i)),\eta_t(\varphi,i)\Big)\Big]+\widetilde{\mathbb{E}}\Big[  d\Big((X_{t}(s,(\psi, j)), \Lambda(t;s,j)),\eta_t(\psi,j)\Big)\Big]\\
& \qquad+\widetilde{\mathbb{E}}\Big[  d\Big((X_{t}(s,(\varphi, i)), \Lambda(t;s,i)),(X_{t}(s,(\psi, j)), \Lambda(t;s,j))\Big)\Big]\to 0,\text{ as } s\to-\infty ,
\end{aligned}
\]
for any \( (\varphi,i),(\psi,j)\in \mathscr{C}_r\times\mathbb{S}\). For any \( -\infty<s \leq t<\infty \), let
\[
P_{s, t}((\varphi,i), \cdot)=\mathbb{P} \circ\left(X_{t}(s, (\varphi,i)),\Lambda(t;s,i)\right)^{-1}(\cdot) 
\]
and 
\[
P_{s, t} f(\varphi,i)=\int_{\mathscr{C}_r\times \mathbb{S}}f(\psi,j)P_{s, t}((\varphi,i),\d \psi\times\d j), \quad f \in \mathscr{B}_{b}(\mathscr{C}_r\times \mathbb{S}).
\]
Then, for any \( -\infty<s \leq u \leq t<\infty \), by the Markov property of \( (X_{t}(s, (\varphi,i)),\Lambda(t;s,i))_{t\geq s}\), we derive that
\[
P_{s, u} \circ P_{u, t}=P_{s, t}
\]
and, by the time-homogeneity, that
\[
P_{s, t}((\varphi,i), \cdot)=P_{0, t-s}((\varphi,i), \cdot) .
\]
Letting $t=0$ in \eqref{Eq:1027:1}, it follows that \( (X_{0}(s, (\varphi,i)),\Lambda(0;s,i))_{s\leq 0} \) converges in probability to \( \eta_{0} \) as \( s \rightarrow-\infty \) by Markov's inequality. Hence,
\[
P_{s, 0}((\varphi,i), \cdot) \rightarrow \mu:=\mathbb{P} \circ \eta_{0}^{-1} \quad \text { weakly as } s \rightarrow -\infty .
\]

In what follows, we shall show that $\mu$ is indeed an invariant probability measure of the FSPDEwM.  Adopting the monotone class argument, it is sufficient to  verify that \eqref{Eq:5} holds for \( f \in Lip_b(\mathscr{C}_r\times\mathbb{S}) \), the set of all bounded Lipschitz functions \( f: \mathscr{C}_r\times\mathbb{S} \mapsto \mathbb{R} \).  If we can show that $(X_{t}(s, (\varphi,i)),\Lambda(t;s,i))_{t\geq s}$ admits the Feller property. Then  for any \( f \in Lip_{b}(\mathscr{C}_r\times\mathbb{S}) \),  \( P_{0, t} f \in C_{b}(\mathscr{C}_r\times\mathbb{S}) \).  Indeed, for any $f\in C_{b}(\mathscr{C}_r\times\mathbb{S})$, it is obvious that $P_{0,t}f(\varphi,i)$ is bounded. Since $\S$ has a discrete metric, it is sufficient to show that $P_{0,t}f(\varphi,i)$ is continuous with respect to $\varphi$.  By virtue of Theorem 5.6 in  \cite{chen2004markov}, we need to prove that
	\begin{equation}\label{Eq:0104}
		\mathbb{W}_2(P_{0,t}((\varphi,k),\cdot),P_{0,t}((\psi,k),\cdot))\rightarrow0\quad {\rm as} \quad \|\varphi-\psi\|_r \rightarrow 0,
	\end{equation}
	where $\mathbb{W}_2(\cdot,\cdot)$ denotes the $L^2$-Wasserstein metric between two probability measures.  The definition of $\mathbb{W}_2(\cdot,\cdot)$ and Theorem \ref{Thm:4.4} imply that \eqref{Eq:0104}  holds.

The definition of weak convergence of probability measures lead to
\[
\mu\left(P_{0, t} f\right)=\lim _{s \rightarrow -\infty} P_{s, 0}\left(P_{0, t} f(\varphi,i)\right)=\lim _{s \rightarrow -\infty} P_{-(t-s), 0} f(\varphi,i)=\mu(f) .
\]
Moreover, noting that \( P_{t} f=P_{0, t} f \) for \( t \geq 0 \),   we get the desired result.

\noindent Step 2: Uniqueness of invariant measure.  Let \( M \geqslant 0 \) be an arbitrary constant. Then, from the invariance of \( \mu \) and Theorem \ref{Thm:4.3}, we have
\[
\begin{aligned}
\int_{\mathscr{C}_r\times\mathbb{S}}f_M (\psi,j)\mu(\mathrm{d} \psi \times \mathrm{d}j) 
& =\int_{\mathscr{C}_r\times\mathbb{S}} P_{t}f_M (\psi,j) \mu(\mathrm{d} \psi \times \mathrm{d}j) \\
& =\int_{\mathscr{C}_r\times\mathbb{S}} \mathbb{E}[f_M(X_{t}(0, (\psi,j)),\Lambda(t;0,j))]\mu(\mathrm{d} \psi \times \mathrm{d}j) \\
& =\int_{\mathscr{C}_r\times\mathbb{S}} (\mathbb{E}\|X_{t}(0, (\psi,j))\|_r\wedge M)\mu(\mathrm{d} \psi \times \mathrm{d}j) \\
& \lesssim1+\int_{\mathscr{C}_r\times\mathbb{S}}(\e^{-\lambda t}\|\psi\|_{r}\wedge M)\mu(\mathrm{d} \psi \times \mathrm{d}j),
\end{aligned}
\]
where $f_M(\psi,j)=\|\psi\|_{r} \wedge M$. Letting first  \( t \rightarrow\infty \) and then letting \( M \rightarrow\infty \), Fatou's lemma  implies that $\int_{\mathscr{C}_r\times\mathbb{S}} \|\psi\|_r\mu\left(\mathrm{d} \psi,\d j\right)<\infty$ . Let \( \nu \in \mathscr{P}(\mathscr{C}_r\times\mathbb{S}) \) be another invariant measure. Then, for any  $ f \in Lip_{b}(\mathscr{C}_r\times\mathbb{S}) $,  \eqref{Eq:4} implies
\[
\begin{aligned}
&\quad|\mu(f)-\nu(f)| \\
& \leq \int_{\mathscr{C}_r\times\mathbb{S}} \int_{\mathscr{C}_r\times\mathbb{S}}\left|P_{t} f(\varphi,i)-P_{t} f\left(\psi,j\right)\right| \mu(\mathrm{d} \varphi,\d i)\nu\left(\mathrm{d} \psi,\d j\right),\\
& \leq \|f\|_{Lip}  \int\int_{(\mathscr{C}_r\times\mathbb{S})^2} \widetilde{\mathbb{E}}\Big[  d\Big((X_{t}(0,(\varphi, i)), \Lambda(t;0,i)),(X_{t}(0,(\psi, j)), \Lambda(t;0,j))\Big)\Big] \mu(\mathrm{d} \varphi,\d i) \nu\left(\mathrm{d} \psi,\d j\right) \\
&\lesssim \mathrm{e}^{-\kappa t}\int\int_{(\mathscr{C}_r\times\mathbb{S})^2} (1+\|\varphi\|_r+\|\psi\|_r)\mu\left(\mathrm{d} \varphi,\d i\right)\nu\left(\mathrm{d} \psi,\d j\right)\\
&\lesssim \mathrm{e}^{-\kappa t}\left(1+\int_{\mathscr{C}_r\times\mathbb{S}}\|\varphi\|_r\mu\left(\mathrm{d} \varphi,\d i\right) +\int_{\mathscr{C}_r\times\mathbb{S}} \|\psi\|_r\nu\left(\mathrm{d} \psi,\d j\right)\right) \rightarrow 0, \quad \text { as } t \rightarrow \infty.
\end{aligned}
\]
Thus, the uniqueness of invariant measure follows. 

\noindent Step 3: Exponential Mixing. By the invariance of \( \mu \) and \eqref{Eq:4}, we obtain that
\begin{align*}
|P_{t}& f(\varphi,i)-\mu(f)| \\
& \leq \int_{\mathscr{C}_r\times\mathbb{S}}\left|P_{t} f(\varphi,i)-P_{t} f\left(\psi,j\right)\right| \mu\left(\mathrm{d} \psi,\d j\right) \\
& \leq\|f\|_{Lip}\int_{\mathscr{C}_r\times\mathbb{S}}\widetilde{\mathbb{E}}\Big[  d\Big((X_{t}(0,(\varphi, i)), \Lambda(t;0,i)),(X_{t}(0,(\psi, j)), \Lambda(t;0,j))\Big)\Big] \mu\left(\mathrm{d} \psi,\d j\right) \\
& \lesssim \mathrm{e}^{-\kappa t}\int_{\mathscr{C}_r\times\mathbb{S}}(1+\|\varphi\|_r+\|\psi\|_r)\mu\left(\mathrm{d} \psi,\d j\right)\\
&\lesssim \mathrm{e}^{-\kappa t} \Big(1+\|\varphi\|_r+\int_{\mathscr{C}_r\times\mathbb{S}}\|\psi\|_r\mu\left(\mathrm{d} \psi,\d j\right)\Big),
\end{align*}
for any $ f \in Lip_b(\mathscr{C}_r\times\mathbb{S})$. Now the proof is complete.
\end{proof}

\begin{remark}\label{RMK:05}
Here, we provide some remarks on the completeness of the space \( L^1(\Omega; \mathscr{C}_r \times \mathbb{S}) \). Suppose that \( \{Y^{(n)}\}_{n \geq 1} = \{(Y^{(n)}_1, Y^{(n)}_2)\}_{n \geq 1} \) is a Cauchy sequence in this space. Then, we have,
\begin{equation}\label{EQ:CAUCHY:02}
    \lim_{n,m \to \infty} \mathbb{E} \left[ \| Y^{(n)}_1 - Y^{(m)}_1 \|_r + \ell(Y^{(n)}_2, Y^{(m)}_2) \right] = 0.
\end{equation}
Since \( (\mathscr{C}_r, \|\cdot\|_r) \) is a Banach space, we can find a unique limit \( Y_1 \in L^1(\Omega; \mathscr{C}_r) \) such that
\[
\lim_{n \to \infty} \mathbb{E} \| Y^{(n)}_1 - Y_1 \|_r = 0.
\]
By \eqref{EQ:CAUCHY:02}, the sequence \( \{ Y^{(n)}_2 \} \) converges in probability. So there exists a subsequence \( \{ Y^{(n_k)}_2 \} \) that converges almost surely to a limit, denoted by \( Y_2 \), under the metric \( \ell(\cdot, \cdot) \). That is,
\[
\mathbbm{1}_{\{ Y^{(n_k)}_2 \neq Y_2 \}} \to 0 \quad \text{as} \quad k \to \infty,
\]
which implies that there exists a sufficiently large \( k \geq 1 \) such that for any \( l \geq 1 \), \( Y^{(n_k)}_2 = Y^{(n_{k+l})}_2 = Y_2 \) almost surely. We conclude that \( Y_2 \in L^1(\Omega; \mathbb{S}) \).
Moreover, for any \( \varepsilon > 0 \), for sufficiently large \(  k,m \), we have
\[
\mathbb{E} \mathbbm{1}_{\{ Y^{(n_k)}_2  \neq Y^{(m)}_2 \}} \leq \varepsilon.
\]
Thus,
\[
\varepsilon \geq \lim_{k \to \infty} \mathbb{E} \mathbbm{1}_{\{ Y^{(n_k)}_2  \neq Y^{(m)}_2  \}} = \mathbb{E} \lim_{k \to \infty} \mathbbm{1}_{\{ Y^{(n_k)}_2  \neq Y^{(m)}_2  \}} = \mathbb{E} \mathbbm{1}_{\{ Y_2  \neq Y^{(m)}_2  \}},
\]
which implies that \( Y^{(m)}_2 \overset{L^1}{\to} Y_2  \in L^1(\Omega; \mathbb{S}) \). Finally, note that we identify \( L^1(\Omega; \mathscr{C}_r) \times L^1(\Omega; \mathbb{S}) \) with \( L^1(\Omega; \mathscr{C}_r \times \mathbb{S}) \), so the completeness of \( L^1(\Omega; \mathscr{C}_r \times \mathbb{S}) \) is established.
\end{remark}

\section{Case  II: infinite state space}\label{sec:5}
In this section, we consider the FSPDEwM \eqref{EQ:FRSPDE:01}  in an infinite state space. In the sequel, we use the same notations as in Section \ref{sec:4}, and assume that (\hyperlink{(B1)}{B1}) and (\hyperlink{(B2)}{B2}) hold.  To address the challenge posed by the infinite state space, we adopt the approach presented in \cite{SHAO2015SPA}. Let us provide a brief description of this approach as follows. First, we divide $\mathbb{S}$ into finite subsets according to ${\alpha(\cdot)}$ given  in (\hyperlink{(A2)}{A2}). Precisely, choose a finite partition $\mathcal{M}$ of ${(-\infty, \alpha_{sup}]}$ of size ${m}$ with ${m \geq 1}$, that is,
\begin{equation}\label{EQ:PARTITION:01}
    	\mathcal{M}:=\left\{-\infty=: i_{0}<i_{1}<\cdots<i_{m}:=\alpha_{sup}\right\} .
\end{equation}
	Corresponding to $\mathcal{M}$, there exists a finite partition of ${\mathbb{S}}$, denoted by ${F:=\left\{F_{1}, \ldots, F_{m}\right\}}$, where
	\[
	F_d=\left\{l \in \mathbb{S} : \alpha(l) \in\left(i_{d-1}, i_{d}\right]\right\}, \quad  d=1, \ldots, m .
	\]
	We assume that each \( F_d \) is non-empty. Otherwise, we can remove certain points from the partition \( \mathcal{M} \) to ensure this condition.
  %Set ${Q^{F}=\left(q_{k l}^{F}\right)}_{m\times m}$ be a new ${Q}$-matrix on the state space ${\{1,2, \ldots, m\}}$ corresponding to ${F}$ defined by	
  Let \( Q^{F} = \left(q_{kl}^{F}\right)_{m \times m} \) be a new \( Q \)-matrix on the state space \( \{1, 2, \ldots, m\} \) corresponding to \( F \), defined by
  \begin{equation}\label{eq:xinleng}
		q_{k l}^{F}=\inf _{j_1 \in F_{k}} \sum_{j_2 \in F_{l}} q_{j_1 j_2},  l>k ; \quad q_{k l}^{F}=\sup _{j_1 \in F_{k}} \sum_{j_2 \in F_l} q_{j_1j_2}, l<k ;\quad  q_{k k}^{F}=-\sum_{l \neq k} q_{k l}^{F}.
	\end{equation}
 Since each ${F_d}$ is non-empty, we obtain $0 \leq q_{k l}^{F} \leq \sup _{k \in \mathbb{S}} q_{k}<\infty, l \neq k$. Hence,  $Q^F$ is well-defined. Moreover,   the consistency of this method on the finite partitions is satisfied (see \cite[Proposition 4.2.]{SHAO2015SPA} for more details). Let
	\begin{equation}\label{eq:xinffu}
		\lambda_1^{F}(d)=\inf _{k \in F_d} \lambda_1(k),\;\alpha^{F}(d)=\sup _{k \in F_d} \alpha(k), \text{ and }  \beta^{F}(d)=\sup _{k \in F_d} \beta(k).
	\end{equation}
	for $d=1, \ldots, m$. Here, it is obvious that $\lambda_1^F(\cdot)$ is  well-defined, while $\alpha^F(\cdot)$ and $\beta^F(\cdot)$ can be shown to be well-defined by (\hyperlink{(B2)}{B2}).
    
 In what follows, we show that \eqref{Eq:1018:2} still works in the case of infinite state space.   Let us consider the difference between two Markov chains starting from different initial times, namely
 \begin{equation}\label{Eq:1020:2}
     \begin{aligned}
 \d\left(\Lambda(t;s_2,i)-\Lambda(t;s_1,i)\right) & =\int_{[0, M]} (h(\Lambda(t-;s_2,i), u)-h(\Lambda(t-;s_1,i)),u)) N(\d t, \d u),
\end{aligned}
 \end{equation}
for any $-\infty<s_1 \leq s_2 \leq t <\infty$. Moreover, we need to introduce a function associated with Eq.\ \eqref{Eq:1020:2}. For any  function \( V:\mathbb{Z}\to \mathbb{R} \), define  \( \mathcal{L} V: \mathbb{S}\times\mathbb{S} \rightarrow \mathbb{R} \) by
\begin{equation}\label{Eq:1020:3}
    \mathcal{L} V(k,l) =  \int_{[0,M]}\left(V(k-l+h(k,u)-h(l,u))-V(k-l)\right)\mathfrak{m}(\d u).
\end{equation}
\begin{lemma}\label{Eq:1020:6}
    Let \(  F \) be a bounded non-negative  function defined on $\mathbb{Z}$. Assume that
\begin{equation}\label{Eq:1020:4}
    \mathcal{L} F(k, l) \leqslant-1, \quad  k \neq l .
\end{equation}
Then for any non-negative \( v \in \mathscr{C}^{1}([0, \infty) )\), where $\mathscr{C}^{1}([0, \infty)) $ is the space of continuously differentiable functions on $[0,\infty)$, we have
\begin{equation}\label{Eq:1020:5}
    \mathbb{E} \int_{s_2}^{t \wedge \tau} v(s)\mathrm{d} s \leqslant v(s_2) \mathbb{E}F(i-\Lambda(s_2;s_1,i))+\mathbb{E} \int_{s_2}^{t \wedge \tau} v^{\prime}(s) F\left(\Lambda(s;s_2,i)-\Lambda(s;s_1,i)\right) \mathrm{d} s,
\end{equation}
where $\tau$ is defined by \eqref{Eq:01041}.
\end{lemma}
\begin{proof}
It follows from \eqref{Eq:1020:3} and \eqref{Eq:1020:4} that $\mathcal{L} F(k, l) \leqslant-1,$ for any $k,l\in \S$. Let \( G(t, k, l)=v(t) F(k-l) \).  Applying \ito's formula to $G(t,  \Lambda(t;s_2,i),  \Lambda(t;s_1,i))$ implies
\begin{equation}
\begin{aligned}
    \mathbb{E}&G\left(t\wedge \tau ,  \Lambda(t\wedge \tau;s_2,i),  \Lambda(t\wedge \tau;s_1,i)\right)\\
    &=\E G(s_2,\Lambda(s_2;s_2,i),\Lambda(s_2;s_1,i))+\mathbb{E}\int_{s_2}^{t\wedge\tau}\Big(v^{\prime}(s)F(\Lambda(s;s_2,i)-\Lambda(s;s_1,i))\\
    &\qquad+v(s)\mathcal{L}F(\Lambda(s;s_2,i),\Lambda(s;s_1,i))\Big)\d s\\
    &\leq v(s_2)\E F(i-\Lambda(s_2;s_1,i))+\mathbb{E}\int_{s_2}^{t\wedge\tau}\Big(v^{\prime}(s)F(\Lambda(s;s_2,i)-\Lambda(s;s_1,i))-v(s)\Big)\d s.
\end{aligned}
\end{equation}
Note that $F$ and $v$ are non-negative. This gives us the required assertion.
\end{proof}
\begin{remark}
    For any function $V$ on $\mathbb{Z}$, by \eqref{EQ:0916:01} and \eqref{Eq:1020:3},
    \[\mathcal{L}V(k,l) = \sum_{m,n\in\mathbb{S}} \left[ V(m - n) - V(k - l) \right] \cdot \mathfrak{m}(\triangle_{km} \cap \triangle_{ln}).\] 
\end{remark}
\begin{remark}
    Applying \eqref{Eq:1020:5} to \( v(t)=\e^{\theta t}~(\theta>0) \), we obtain
\[
\begin{aligned}
    \theta^{-1} \mathbb{E}\left(\e^{\theta(t \wedge \tau)}-\e^{\theta s_2}\right) &\leqslant \e^{\theta s_2}\E F(i-\Lambda(s_2;s_1,i))\\
&\quad +\theta\mathbb{E}\int_{s_2}^{t\wedge\tau}\e^{\theta s}F(\Lambda(s;s_2,i)-\Lambda(s;s_1,i))\d s\\
&\leq \|F\|_{\infty}\e^{\theta s_2}+\|F\|_{\infty} \mathbb{E}\left(\e^{\theta(t \wedge \tau)}-\e^{\theta s_2}\right),
\end{aligned}
\]
where $\|F\|_{\infty}:=\sup_{k\in\S}F(k)<\infty$. Thus, by Fatou's lemma, we have
\[
\mathbb{E} \e^{\theta \tau}=\mathbb{E} \left(\lim_{t\to\infty}\e^{\theta (t\wedge\tau)}\right)\leq \liminf_{t\to\infty}\mathbb{E} \e^{\theta (t\wedge\tau)} \leqslant \e^{\theta s_2}\left(1+\frac{\|F\|_{\infty}}{\theta^{-1}-\|F\|_{\infty}}\right)=\frac{\e^{\theta s_2}}{1-\theta\|F\|_{\infty}}, 
\]
for $0<\theta<1/\|F\|_{\infty}<\infty $. Choosing a $\theta^F$ such that $0<\theta^F<1/\|F\|_{\infty}$,  by Markov's inequality, we have
\begin{equation}\label{Eq:1019:1}
   \mathbb{P}(\tau>t)=\mathbb{P}(\e^{\theta^F\tau}>\e^{\theta^F t})\leq  \frac{\E{\e^{\theta^F\tau}}}{\e^{\theta^F t}}\lesssim  \e^{-\theta^F(t-s_2)},\; t\geq s_2.
\end{equation}
\end{remark}

\begin{theorem}\label{thm:matrix}
		Let (\hyperlink{A1}{A1})-(\hyperlink{A3}{A3}), (\hyperlink{B1}{B1}) and (\hyperlink{B2}{B2}) hold, and let assumptions of Lemma \ref{Eq:1020:6} hold. For the partition $\mathcal{M}$ given in \eqref{EQ:PARTITION:01},  assume further that $${\mathcal{A}}^F:=-\left(Q^{F}+\operatorname{diag}\left(\alpha^{F}(1)-2\lambda_1^F(1)+L, \ldots, \alpha^{F}(m)-2\lambda_1^F(m)+L\right)\right) H_{m} $$ is a non-singular $M$-matrix, where $H_m$ is defined by \begin{equation}\label{eq:mingji}
			H_{m}=\left(\begin{array}{ccccc}
				1 & 1 & 1 & \cdots & 1 \\
				0 & 1 & 1 & \cdots & 1 \\
				\vdots & \vdots & \vdots & \cdots & \vdots \\
				0 & 0 & 0 & \cdots & 1
			\end{array}\right)_{m \times m},
		\end{equation} 
	so $\eta^{F}=\left(\eta^{F}(1), \ldots,\eta^{F}(m)\right)^{\top} :=({\mathcal {A}}^F)^{-1}\vec{1}\gg 0 $.  Let \( \xi^{F}=H_{m} \eta^{F} \). If $$  (\beta^F(k)+L)\xi^F(k)<1,\; \text{for all}\; 1\leq k\leq m,$$
		then the FSPDEwM \eqref{EQ:FRSPDE:01} has a unique invariant measure \( \mu \in \mathscr{P}(\mathscr{C}_r) \), which is also exponentially mixing.\end{theorem}
	\begin{proof}
	 %Note that $\xi^F=H_m\eta^F$, then 
	%\[
	%\xi^{F}({d})=\eta^{F}(d)+\cdots+\eta^{F}(m), \quad  d=1, \ldots, m .
	%\]
	%Hence, ${\xi^{F}({d+1})<\xi^{F}({d}), d=1, \ldots, m-1}$, and ${\xi^{F} \gg 0}$.
 Note that $\xi^F = H_m \eta^F$. Then, we have
\[
\xi^{F}(d) = \eta^{F}(d) + \cdots + \eta^{F}(m), \quad d = 1, \ldots, m.
\]
Thus, we obtain $\xi^{F}(d+1) < \xi^{F}(d)$ for $d = 1, \ldots, m-1$, and $\xi^{F} \gg 0$.
Let us extend the vector ${\xi^{F}}$ to a vector on ${\mathbb{S}}$ by setting ${\xi({k})=\xi^{F}({d})}$, if ${k \in F_d}$. Thus, by \eqref{eq:xinleng}, we obtain 
	\begin{equation}\label{eq:aixipo}
		\begin{aligned}
			(Q \xi)(k) &=\sum_{ l \neq k} q_{k l}\left(\xi(l)-\xi(k)\right)=\sum_{l \notin F_d} q_{k l}\left(\xi(l)-\xi(k)\right) \\
			&=\sum_{j<d}\Big(\sum_{l \in F_{j}} q_{k l}\Big)\left(\xi^{F}(j)-\xi^{F}({d})\right)+\sum_{j>d}\Big(\sum_{l \in F_{j}} q_{k l}\Big)\left(\xi^{F}(j)-\xi^{F}({d})\right) \\
   &\leq \sum_{j<d}\Big(\sup_{k\in F_d}\sum_{l \in F_{j}} q_{k l}\Big)\left(\xi^{F}(j)-\xi^{F}({d})\right)+\sum_{j>d}\Big(\inf_{k\in F_d}\sum_{l \in F_{j}} q_{k l}\Big)\left(\xi^{F}(j)-\xi^{F}({d})\right) \\
			& = \sum_{j<d} q_{d j}^{F}\left(\xi^{F}(j)-\xi^{F}(d)\right)+\sum_{j>d} q_{d j}^{F}\left(\xi^{F}(j)-\xi^{F}(d)\right)=\left(Q^{F} \xi^{F}\right)(d),
		\end{aligned}
	\end{equation}
	for any ${k \in F_d}$. Let ${h: \mathbb{S} \rightarrow\{1, \ldots, m\}}$ be a map defined by ${h(k)=d}$ if ${k\in F_d}$. Hence, $\xi(k)=\xi^F(h(k))$, $\lambda_1(k)\geq \lambda_1^F(h(k))$, ${\alpha(k) \leq \alpha^{F}(h(k))}$, ${\beta(k) \leq \beta^{F}(h(k))}$ and $(Q \xi)(k)\leq \left(Q^{F} \xi^{F}\right)(h(k))$  for any ${k \in \mathbb{S}}$.

 In what follows, we shall show that \eqref{EQ:1013:02} and   \eqref{Eq:1018:1} also hold in the case of infinite state space.   Noting that $\mathcal{A}^F$ is a non-singular $M$-matrix, we have 
\[
(Q^F\xi)(h(k))+(\alpha^F(h(k))-2\lambda_1^F(h(k))+L)\xi^F(h(k))=-1,
\]
for any $k\in \S$. Let $\lambda\in (0,2r)$ be a constant  to be determined later. 
%By (\hyperlink{B2}{B2}), \eqref{Eq:1012:1} still holds.
According to (\hyperlink{B2}{B2}),   \eqref{Eq:1012:1} remains valid.
Thus, applying  the generalized \ito~ formula and using \eqref{Eq:6}, we have 
\begin{align*}
		\e ^{\lambda t} &\E\left [\|X(t)\|_H^2\xi(\Lambda(t))\right] \\
 &\leq \e ^{\lambda s}\E[\|X(s)\|_H^2\xi(\Lambda(s))] +C(\varepsilon)\int_s^t\e^{\lambda u}\xi(\Lambda(u))\d u+ \lambda\E \int_{s}^{t}\e ^{\lambda u}  \|X(u)\|_H^2\xi(\Lambda(u))\d u\\
 &\quad+\E \int_{s}^{t}\e ^{\lambda u}  \big( (Q\xi)(\Lambda(u))+(\alpha(\Lambda(u))-2\lambda_1(\Lambda(u))+L+2\varepsilon)\xi(\Lambda(u))\big)\|X(u)\|_H^2\d u \\
 &\quad+\E \int_{s}^{t}\int_{-\infty}^0\e ^{\lambda u}(\beta(\Lambda(u))+L+\varepsilon)\xi(\Lambda(u))\|X(u+\theta)\|_H^2\rho(\d \theta)\d u\\
& \leq \e ^{\lambda s}\E[\|X(s)\|_H^2\xi^F(h(\Lambda(s)))] +\frac{C(\varepsilon)\xi_{max}^F}{\lambda}\e^{\lambda t}+\lambda\xi_{max}^F \E \int_{s}^{t}\e ^{\lambda u}  \|X(u)\|_H^2\d u\\
 &\quad+\E \int_{s}^{t}\e ^{\lambda u}  \Big( (Q^F\xi)(h(\Lambda(u)))+\big(\alpha^F(h(\Lambda(u)))-2\lambda_1^F(h(\Lambda(u)))\\
&\qquad\qquad\qquad\quad+L+2\varepsilon\big)\xi^F(h(\Lambda(u)))\Big)\|X(u)\|_H^2\d u \\
 &\quad+\E \int_{s}^{t}\int_{-\infty}^0\e ^{\lambda u}(\beta^F(h(\Lambda(u)))+L+\varepsilon)\xi^F(h(\Lambda(s)))\|X(u+\theta)\|_H^2\rho(\d \theta)\d u\\
 &\leq \xi_{max}^F\e ^{\lambda s}\|\varphi\|_r^2+\frac{C(\varepsilon)\xi_{max}^F}{\lambda}\e^{\lambda t}+\left((\lambda+2\varepsilon)\xi_{max}^F-1\right)\E \int_{s}^{t}\e ^{\lambda u}\|X(u)\|_H^2\d u\\
 &\quad+\left(M^F+\varepsilon\xi_{max}^F\right)\E \int_{s}^{t}\int_{-\infty}^0\e ^{\lambda u}\|X(u+\theta)\|_H^2\rho(\d \theta)\d u\\
 &\leq \widetilde{C}(\varepsilon) \e ^{\lambda s}\|\varphi\|_r^2+\frac{C(\varepsilon)\xi_{max}}{\lambda}\e^{\lambda t}+K^F(\lambda)\E \int_{s}^{t}\e ^{\lambda u}\|X(u)\|_H^2\d u,
\end{align*}
where $M^F=\max_{1\leq d\leq m}(\beta^F(d)+L)\xi^F(d)$, $\widetilde{C}^F(\varepsilon)=\left(1+\frac{\rho^{(2r)}(\beta_{max}^F+L+\varepsilon)}{2r-\lambda}\right)\xi_{max}^F$ and 
\[
K^F(\lambda)=(\lambda+2\varepsilon)\xi_{max}^F-1+(M^F+\varepsilon\xi_{max}^F)\int_{-\infty}^{0}\e ^{-\lambda \theta}\rho(\d \theta).
\]
Since $(\beta^F(d)+L)\xi^F(d)<1$ for any $1\leq d\leq m$, then $M^F<1$. Moreover, $K^F(0)=3\varepsilon\xi_{max}^F-1+M^F$, so we can choose $\varepsilon>0$ and $\lambda\in (0,2r)$ to be sufficiently small such that $K^F(\lambda)\leq 0$, which implies that \eqref{Eq:1015:1} holds again. Similarly, we can also prove that there exists a constant $\lambda^F>0$ such that   \eqref{EQ:GAMMA:01} holds too. 
%Based on the above results, we derive \eqref{EQ:1013:02} and   \eqref{Eq:1018:1} by carrying out an argument similar to Theorem \ref{Thm:4.3} and Theorem \ref{Thm:4.4}, respectively. 
Based on the above results, we derive \eqref{EQ:1013:02} and \eqref{Eq:1018:1} by employing arguments analogous to those in Theorem \ref{Thm:4.3} and Theorem \ref{Thm:4.4}, respectively.

Furthermore, by the strong Markov property, we obtain that for any different initial times $-\infty<s_1 \leq s_2 \leq t <\infty$,
\begin{align*}
\widetilde{\mathbb{E}}&\Big[  d\Big(\big(X_{t}(s_1,(\varphi, i)), \Lambda(t;s_1,i)\big),\big(X_{t}(s_2,(\varphi, i)), \Lambda(t;s_2,i)\big)\Big)\Big] \\
&=  \widetilde{\mathbb{E}}\left[\widetilde{\mathbb{E}}\left[\left(\|X_{t}(s_1,(\varphi, i))-X_{t}(s_2,(\varphi, i))\|_{r}+\ell\left(\Lambda(t;s_1,i), \Lambda(t;s_2,i)\right)\right)\1_{\{\tau \leq (t+s_2) / 2\}} \mid \widetilde{\mathscr{H}}_{\tau}\right]\right] \\
&\quad +\widetilde{\mathbb{E}}\left[(\left\|X_{t}(s_1,(\varphi, i))-X_{t}(s_2,(\varphi, i))\right\|_{r}+\ell\left(\Lambda(t;s_1,i), \Lambda(t;s_2,i)\right))\1_{\{\tau>(t+s_2) / 2\}}\right] \\
& \leq \widetilde{\mathbb{E}}\left[ \widetilde{\mathbb{E}}\left[\left\|X_{t}(\tau,(X_{\tau}(s_1,(\varphi, i)), \Lambda(\tau;s_1,i))-X_{t}(\tau,(X_{\tau}(s_2,(\varphi, i))\Lambda(\tau;s_2,i))\right\|_{r}\right]\1_{\{\tau \leq (t+s_2) / 2\}}\right] \\
&\quad+\Big[\widetilde{\mathbb{E}}\big(1+\left\|X_{t}(s_1,(\varphi, i))-X_{t}(s_2,(\varphi, i))\right\|_{r}\big)^{2}\Big]^{1 / 2} [\mathbb{P}(\tau>(t+s_2) / 2)]^{1 / 2}  \\
& \lesssim \widetilde{\mathbb{E}}\left(\e^{-\lambda^F(t-\tau)/2} \left\|X_{\tau}(s_1,(\varphi, i))-X_{\tau}(s_2,(\varphi, i))\right\|_{r}\1_{\{\tau \leq (t+s_2) / 2\}}\right) \\
&\quad+\left(1+ \widetilde{\mathbb{E}}\left\|X_{t}(s_1,(\varphi, i))-X_{t}(s_2,(\varphi, i))\right\|_{r}^{2}\right)^{1 / 2}\mathrm{e}^{-\theta^F (t-s_2) / 4} \\
& \lesssim \widetilde{\mathbb{E}}\left(\left\|X_{\tau}(s_1,(\varphi, i))\right\|_{r}+\left\|X_{\tau}(s_2,(\varphi, i))\right\|_{r}\right) \e^{-\lambda^F (t-s_2) / 4} \\
&\quad+\left(1+ \widetilde{\mathbb{E}}\left\|X_{t}(s_1,(\varphi, i))\right\|_{r}^{2}+ \widetilde{\mathbb{E}}\left\|X_{t}(s_2,(\varphi, i))\right\|_{r}^{2}\right)^{1 / 2} \mathrm{e}^{-\theta^F (t-s_2) / 4}\\
& \lesssim\left(1+\|\varphi\|_{r}\right) \e^{-\kappa^F (t-s_2) / 4},
\end{align*}
where we have used \eqref{Eq:1019:1} and  $\kappa^F=\lambda^F\wedge\theta^F$. 
Then, by following a similar argument just as in Theorem \ref{THM:MAIN:01}, we can conclude the proof.
 	\end{proof}

\begin{appendix}
\renewcommand{\theequation}{\Alph{section}.\arabic{equation}}
 \section{}\label{Appendix}
 In this appendix, we establish the existence and uniqeness of mild solutions about FSPDEs without Markovian switching. For this aim, consider the following semi-linear functional stochastic partical differential equation
\begin{equation}\label{EQ:GENERAL:01}
\d X(t) =[AX(t)+B(X_t)]\d t+\Sigma(X_t)\d W(t),\quad X_0=\varphi,
\end{equation}
where $A$ is a linear operator with domain $\mathscr{D}(A)$ and generating a contractive $C_0$-semigroup, $B:\mathscr{C}_r\to H$, and $\Sigma:\mathscr{C}_r\to L_2(U;H)$ are measurable mappings.
%Before establishing the existence and uniqueness of the mild solution to the FSPDEwM   \eqref{EQ:FRSPDE:01}, we need to introduce a family of stochastic differential equations. For each $k\in\S$, consider the following semi-linear functional stochastic  partial differential equation, namely
%\begin{equation}\label{Eq:1020:1}\d X^{(k)}(t)=[A(k) X^{(k)}(t)+b(X_{t}^{(k)}, k)] \d t+\sigma(X_{t}^{(k)}, k) \d W(t).\end{equation}
In what follows, we prove the well-posedness of Eq.\ \eqref{EQ:GENERAL:01} under global (local) Lipschitz condition, respectively.  First, we propose a basic assumption:

\noindent(\hypertarget{(H0)}{H0})  $ (-A, \mathscr{D}(A)) $ is a self-adjoint operator with discrete spectrum
\[ 
0<\lambda_{1} \leq \lambda_{2} \leq \cdots \leq \lambda_{n} \leq \cdots,
 \]
where multiplicities are counted, and such that $ \lambda_{n} \uparrow \infty, n \rightarrow \infty $. Furthermore, $A $ can generate a $ C_0 $-semigroup \( (\e^{t A{}})_{t \geq 0} \) satisfying \( \|\e^{t A}\| \leqslant \e^{-\lambda_{1} t}, t \geqslant 0 \). 
\begin{lemma}\label{Lem:1113}
    %Fixing $k\in \S$, let (\hyperlink{A1}{A1}) holds, and let $b(\cdot,k)$ and $\sigma_1(\cdot,k)$ be Lipschitz continuous. Then for any $p>2$ and initial data $\varphi\in \mathscr{C}_r$, Eq. \eqref{Eq:1020:1} admits a unique mild solution $(X^{(k)}(t))_{t\geq0}$, which satisfies 
    Assume that  (\hyperlink{H0}{H0}) holds, and let \(B(\cdot)\) and \( \Sigma(\cdot) \) satisfy the global Lipschitz condition. Then for any initial data \( \varphi \in \mathscr{C}_r \), Eq.\ \eqref{EQ:GENERAL:01} admits a unique mild solution \( (X(t))_{t \geq 0} \).
\end{lemma}
\begin{proof}
    Following the approach outlined in \cite[Theorem 5.1]{BAO2019SPA}, we prove this lemma by using the classical fixed point theorem. It suffices to show that Eq.\ \eqref{EQ:GENERAL:01} admits a unique mild solution \( (X(t))_{t\in [0,T]} \)  for any fixed $T > 0$. Let
\begin{align*}
\mathscr{H}_{T}&=\bigg\{U=(U(t))_{t \in(-\infty, T]}\mid U  \text{~is a continuous adapted process on~} {H}  \text{~with~}U_{0}=\varphi
\text { and } \\
& \qquad\qquad \mathbb{E}\Big(\sup _{t \in(-\infty, T]}\mathrm{e}^{4r t}\|U(t)\|_{H}^{4}\Big)<\infty\bigg\}.
\end{align*}
Then \( \mathscr{H}_{T} \) is a complete metric space with
\[
\rho(U, V):=\|U-V\|_{\mathscr{H}_{T}}:=\bigg[\mathbb{E}\Big(\sup _{t \in[0, T]}\mathrm{e}^{4r t}\|U(t)-V(t)\|_H^{4}\Big)\bigg]^{\frac{1}{4}} .
\]
Observe that the metric \( \rho \) is equivalent to the metric below
\[
\rho_{0}(U,V):=\|U-V\|_{\mathscr{H}_{T}^{0}}:=\bigg[\mathbb{E}\Big(\sup _{t \in[0, T]}\|U(t)-V(t)\|_H^{4}\Big)\bigg]^{\frac{1}{4}} .
\]
We now prove the map $\mathscr{K}$ define as follow, for any $0\leq t \leq T$, $U \in \mathscr{H}_{T}$,
\begin{equation}\label{EQ:MILDMAP:01}
    \mathscr{K}(U)(t):=\mathrm{e}^{t A} \varphi(0)+\int_{0}^{t} \mathrm{e}^{(t-s) A} B\left(U_{s}\right) \mathrm{d} s+\int_{0}^{t} \mathrm{e}^{(t-s) A} \Sigma\left(U_{s}\right) \mathrm{d} W(s)
\end{equation}
mapping from \( \mathscr{H}_{T} \)  to \( \mathscr{H}_{T} \). To be specific, for any \( U \in \mathscr{H}_{T} \), by differentiating both sides of \eqref{EQ:MILDMAP:01} we have
\[
\mathrm{d}\mathscr{K}(U)(t)=[A\mathscr{K}(U)(t)+B\left(U_{t}\right)] \mathrm{d} t+\Sigma\left(U_{t}\right) \mathrm{d} W(t).
\]
According to \cite[Proposition 2.1.4]{liu2005stability}, we obtain from (\hyperlink{(H0)}{H0}) that
\begin{equation}\label{Eq:1}
    \langle x, A x\rangle_H\leq-\lambda_{1}\|x\|_{H}^{2}, \quad x \in \mathscr{D}(A) .
\end{equation}
Recall that for any $\varepsilon>0$,
\begin{equation}\label{eq:shoushi}
2\langle a,b \rangle_H\leq \varepsilon\|a\|_H^2+\frac{1}{\varepsilon}\|b\|_H^2
\end{equation}
  for any $a,b\in H$. Hence, applying \ito's formula to \( \|\mathscr{K}(U)(t)\|_{H}^{2} \)  derives that for any $\varepsilon>0$,
\[
\begin{aligned}
 \mathrm{d}\|\mathscr{K}(U)(t)\|_H^{2}
& =  2\langle\mathscr{K}(U)(t), A\mathscr{K}(U)(t)+B\left(U_{t}\right)\rangle  \mathrm{d} t +\left\|\Sigma\left(U_{t}\right)\right\|_{ L_2}^{2} \mathrm{d} t+\mathrm{d} M(t) \\
&\leq \varepsilon \|\mathscr{K}(U)(t)\|_H^{2} \mathrm{~d} t+\Big(\frac{1}{\varepsilon}\|B\left(U_{t}\right)\|_H^2+\left\|\Sigma\left(U_{t}\right)\right\|_{ L_2}^{2} \Big)\d t+\mathrm{d} M(t),
\end{aligned}
\]
where 
\[
M(t):=2 \int_{0}^{t}\left\langle\mathscr{K}(U)(s),\Sigma\left(U_{s}\right)\mathrm{d} W(s)\right\rangle
\]
is a (local) martingale. Letting $\varepsilon= \frac{1}{2 \sqrt{5} T}$ and using the linear growth property of $B(\cdot)$ and $\Sigma(\cdot)$ implies 
\begin{equation}\label{Eq:1111:7}
    \begin{aligned}
   \|\mathscr{K}(U)(t)\|_H^{2} &\leq \|\varphi(0)\|_H^2+ C_1T+\frac{1}{2 \sqrt{5} T}\int_0^t\|\mathscr{K}(U)(s)\|_H^{2} \mathrm{~d} s  \\
   &\quad +C_{2}(1+T)\int_0^t\left\|U_{s}\right\|_{r}^{2} \mathrm{d} s+M(t),
\end{aligned}
\end{equation}
where $C_1$ and $C_2$ are some positive constants.  By BDG's inequality and Young's inequality, we obtain that there exist  two positive constants \( C_{3},C_4 \) such that
\begin{equation}\label{Eq:1111:8}
    \begin{aligned}
    \mathbb{E}\Big(\sup _{t \in[0, T]} M(t)^{2}\Big)
    &\leq 16\E\left(\int_0^T\|\mathscr{K}(U)(t)\|_H^2\|\Sigma(U_t)\|_{L_2}^2\d t\right)\\
    &\leq \frac{1}{20}\|\mathscr{K}(U)\|_{\mathscr{H}_{T}^{0}}^{4}+C_3T+C_4T\|U\|_{\mathscr{H}_{T}}^{4}
\end{aligned}
\end{equation}
Furthermore, taking the expectation on both sides of \eqref{Eq:1111:7} and using \eqref{Eq:1111:8} yields 
\begin{equation*}\label{EQ:CONTRACTION:01}
\begin{split}
\|\mathscr{K}(U)\|_{\mathscr{H}_{T}^{0}}^{4} &=\mathbb{E}\Big(\sup _{t \in[0, T]}|\mathscr{K}(U)(t)|^{4}\Big) \\
& \leq 5\|\varphi\|_r^4+5C_1^2T^2+\frac{1}{4}\|\mathscr{K}(U)\|_{\mathscr{H}_{T}^{0}}^{4}+5 C_{1}^{2}(1+T)^{2} T^{2}\|U\|_{\mathscr{H}_{T}}^{4}+5 \mathbb{E}\Big(\sup _{t \in[0, T]} M(t)^{2}\Big) \\
& \leq 5\|\varphi\|_r^4+5C_3T+5C_1^2T^2 +\frac{1}{2}\|\mathscr{K}(U)\|_{\mathscr{H}_{T}^{0}}^{4}+C(T)\|U\|_{\mathscr{H}_{T}}^{4},
\end{split}
\end{equation*}
where $C(T)=5 C_{1}^{2}(1+T)^{2} T^{2}+5C_4T$. Thus,
\begin{equation}\label{EQ:WELL:01}
    \|\mathscr{K}(U)\|_{\mathscr{H}_{T}}^{4} \leq 10\e^{4rT}(\|\varphi\|_r^4+C_3T+C_1^2T^2 )+2\e^{4rT}C(T)\|U\|_{\mathscr{H}_{T}}^{4},
\end{equation}
from which we can deduce that $\mathscr{K}$ is a map from  \( \mathscr{H}_{T} \)  to \( \mathscr{H}_{T} \). 

In the sequel, we need to prove the well-posedness by using of the fixed point theorem. It is sufficient to find \( T_{0}>0 \) independent of \( \varphi \) such that  the map \( \mathscr{K} \) is contractive in \( \mathscr{H}_{T_0} \) since the well-posedness on $[0,T]$ can be done analogously on \( \left[T_{0}, 2 T_{0}\right], \ldots, [\lfloor T/T_0 \rfloor T_0,(\lfloor T/T_0 \rfloor+1) T_0\wedge T] \). Indeed, for any \( U, V \in \mathscr{H}_{T} \), due to  \eqref{EQ:MILDMAP:01} we have
\[
\mathrm{d}[\mathscr{K}(U)(t)-\mathscr{K}(V)(t)]=[A(\mathscr{K}(U)(t)-\mathscr{K}(V)(t))+B\left(U_{t}\right)-B\left(V_{t}\right)] \mathrm{d} t+[\Sigma\left(U_{t}\right)-\Sigma\left(V_{t}\right)] \mathrm{d} W(t).
\]
Similar to  \eqref{EQ:WELL:01}, we can show that
\[
\|\mathscr{K}(U)-\mathscr{K}(V)\|_{\mathscr{H}_{T}}^{4} \leq 2\mathrm{e}^{4 r T}\widetilde{C}(T)\|U-V\|_{\mathscr{H}_{T}}^{4},
\]
where $\widetilde{C}(T)$ is a positive constant dependent of $T$. Therefore, by taking \( T_{0}>0 \) such that \( 2 \mathrm{e}^{4 r T_{0}}\widetilde{C}(T_0)<1 \), we get our desired assertion. The proof is complete.
\end{proof}
To proceed, let us  impose the following assumptions for $B(\cdot)$ and $\Sigma(\cdot)$.

\noindent (\hypertarget{(H1)}{H1}) Both \( B \) and \( \Sigma \) satisfy the local Lipschitz condition, that is, for any \( n>0 \), there exists a \( C_{n} \) such that
\[
\|B(\mathbf{x})-B(\mathbf{y})\|_H \vee\|\Sigma(\mathbf{x})-\Sigma(\mathbf{y})\|_{L_2} \leq C_{n}\|\mathbf{x}-\mathbf{y}\|_{r}
\]
for those \( \mathbf{x}, \mathbf{y} \in \mathscr{C}_{r} \) with \( \|\mathbf{x}\|_{r} \vee\|\mathbf{y}\|_{r} \leq n \).

\noindent(\hypertarget{(H2)}{H2}) There exist constants \( \alpha \in \mathbb{R} \), \( \beta>0 \)  and a probability measure \( \rho \in \mathscr{P}_{2 r}\left(\mathbb{R}^{-}\right) \) such that for any  \( \mathbf{x} , \mathbf{y} \in \mathscr{C}_{r} \), 
	\[ 
	\begin{aligned} 
		& 2\langle\mathbf{x}(0)-\mathbf{y}(0), B(\mathbf{x})-B(\mathbf{y})\rangle_{H}\leqslant  \alpha\|\mathbf{x}(0)-\mathbf{y}(0)\|_{H}^{2}+\beta \int_{-\infty}^{0} \| \mathbf{x}(\theta)-\mathbf{y}(\theta)\|_{H}^{2} \rho(\d \theta).
		\end{aligned} 
	 \]
 \noindent(\hypertarget{(H3)}{H3})   There exists  a constant \( L>0 \) such that for any  \( \mathbf{x} , \mathbf{y} \in \mathscr{C}_{r} \),
	\[
	\left\|\Sigma(\mathbf{x})-\Sigma(\mathbf{y})\right\|_{L_2}^{2}  
		 \leq L\left(\|\mathbf{x}(0)-\mathbf{y}(0)\|_{H}^{2}+\int_{-\infty}^{0}\|\mathbf{x}(\theta)-\mathbf{y}(\theta)\|_{H}^{2} \rho(\d \theta)\right),
	\]
    where $\rho$ is determined in (\hyperlink{(H2)}{H2}).
\begin{lemma}\label{Lem:1024:1}
   Assume that  (\hyperlink{H1}{H1})-(\hyperlink{H3}{H3}) hold. Then for any initial data $\varphi\in \mathscr{C}_r$, Eq.\ \eqref{EQ:GENERAL:01} admits a unique  mild solution.
\end{lemma}
\begin{proof}
%Similarly to the proof of \cite[Theorem 3.2]{Wu2017JDE}, we need to prove Eq.\ \eqref{Eq:1020:1} has a unique maximal local solution by using the truncation argument, and then prove that the explosion time is $\infty$ a.s..  
Following the approach in the proof of \cite[Theorem 3.2]{Wu2017JDE}, we first need to demonstrate that Eq.\ \eqref{EQ:GENERAL:01} has a unique maximal local mild solution using a truncation argument. We then show that the explosion time is almost surely infinite, which ensures that Eq.\ \eqref{EQ:GENERAL:01} indeed possesses a unique global mild solution.

 %For any fixed $\varphi\in\mathscr{C}_r$,  we can choose a positive number  $n_0$  sufficient large such that $\|\varphi\|_r < n_0$. 
 
 For any  $n \geq  1$, let us define truncation the functionals \( B_{n} \) and \( \Sigma_{n} \) as follows:
		\[
		\begin{array}{c}
			B_{n}(\x)=\left\{\begin{array}{ll}
				B(\x), & \|\x\|_{r} \leq n, \\
				B\left(n \x/ \|\x\|_{r}\right), & \|\x\|_{r}>n,
			\end{array}\right. \quad
			\Sigma_{n}(\x)=\left\{\begin{array}{ll}
				\Sigma(\x), & \|\x\|_{r} \leq n, \\
				\Sigma\left(n \x/ \|\x\|_{r}\right), & \|\x\|_{r}>n.
			\end{array}\right.
		\end{array}
		\]
	It is obvious that  \( B_{n} \) and \( \Sigma_{n} \) satisfy the global Lipschitz condition and the linear growth condition. Consider
	\begin{equation}\label{eq:nicai}
			\d X^{(n)}(t) =[AX^{(n)}(t)+B(X_t^{(n)})]\d t+\Sigma(X_t^{(n)})\d W(t),\quad X_0^{(n)}=\varphi.
		\end{equation}
 Hence by Lemma \ref{Lem:1113}, there exist a unique mild solution $(X^{(n)}(t))_{t\geq 0}$ satisfying  \eqref{eq:nicai}. Define a stopping times sequence
		\[
		 \sigma_{n}=\inf \left\{t \geq 0 ,\|X_t^{(n)}\|_r\geq n\right\},
		\]
	%\rho_{n}=\inf \left\{t \in[0,\tau_e) ,\|X^{(n)}(t)\|_H\geq n,\right\} \quad
    with the usual convention \( \inf \varnothing=\infty \).  It is not difficult to show that
 \[
X^{(n+1)}(t)=X^{(n)}(t),\quad \text{if }~0\leq t\leq \sigma_n.
 \]
This implies that	\(\left\{\sigma_{n}\right\}_{n \geq 1} \) is a non-decreasing stopping times sequence and \( \sigma_{n} \rightarrow \sigma_{\infty} \) a.s., as \( n \rightarrow \infty \).	Define \( X(t)\), \(0\leq t \leq \sigma_{\infty} \), by
\[
X(t)=X^{(n)}(t),\quad t\in [\sigma_{n-1},\sigma_n),\;n\geq 1.
\]
 By a standard procedure, we can show \( X(t)\), \(0\leq t \leq \sigma_{\infty} \), is also the unique maximal local mild solution.
		
		To show that this mild solution is global, it is sufficient to prove that \( \sigma_{\infty}=\infty \) a.s.  This is equivalent to proving that for any \( T>0\), \(\mathbb{P}\left(\sigma_{n} \leq T\right) \rightarrow 0 \) as \( n \rightarrow \infty \).  By the definitions of $\sigma_n$ and $X(t)$, we have $\|X_{\sigma_n}\|_{r} = n$, which implies 
		\begin{equation}\label{eq:xiawu}
			\begin{aligned}
				n^2\P(\sigma_n\leq T)&=\E\left(\|X_{\sigma_{n}}\|_r^2 \1_{\{\sigma_n\leq T \}} \right)\leq \E\|X_{T\wedge\sigma_n}\|_r^2\\
				&\leq \|\varphi\|_r^2+\E\Big(\sup _{0< t\leq T\wedge\sigma_n} \e^{2 r t}\left\|X(t)\right\|_H^{2}\Big).
			\end{aligned}
		\end{equation}
 Set $H(t):=\E\left(\sup _{0< s\leq t\wedge\sigma_n} \e^{2 r s}\left\|X(s)\right\|_H^{2}\right)$, $0\leq t\leq T$.  In what follows, we shall estimate $H(t)$.  It follows from (\hyperlink{H2}{H2}) and (\hyperlink{H3}{H3}) that there exists a positive constant $C$ such that 
  \[
\begin{aligned}
   2&\langle\mathbf{x}(0), A\x(0)+B(\mathbf{x})\rangle+\|\Sigma(\mathbf{x})\|_{L_2}^2\\
   &=2\langle \mathbf{x}(0),A\mathbf{x}(0)\rangle+2\langle\mathbf{x}(0)-\textbf{0}_H,B(\mathbf{x})-B(\mathbf{0})\rangle+2\langle \mathbf{x}(0),B(\mathbf{0})\rangle+\|\Sigma(\mathbf{x})-\Sigma(\mathbf{0})+\Sigma(\mathbf{0})\|_{L_2}^2\\
   &\leq  C+(\alpha+1+2L)\|\mathbf{x}(0)\|_H^2+(\beta +2L)\int_{-\infty}^0\|\mathbf{x}(\theta)\|_H^2 \rho(\mathrm{d} \theta) 
\end{aligned}
  \]
  Thus, applying Itô's formula to \( \e^{2 r t}\left\|X(t)\right\|_H^{2}\) yields for any \( t \geq 0 \),
		\begin{equation*}\label{eq:shuijiao}
			\begin{aligned}
				\E&\left(\sup _{0< s\leq t\wedge\sigma_n} \e^{2 r s}\left\|X(s)\right\|_H^{2}\right)\\
				&\leq \|\varphi(0)\|_H^2+2r\E\int_{{0}}^{t\wedge\sigma_n}\e^{2rs}\|X(s)\|_H^2\d u\\
               &\quad+ \E\Bigg(\sup _{0< s\leq t\wedge\sigma_n}\int_{{0}}^s\e^{2ru}\Big(2\langle X(u),AX(u)+B(X_u) \rangle+\|\Sigma(X_u)\|_{L_2}^2\Big)\d u\Bigg)\\
    &\quad+ 2\E\left(\sup _{0< s\leq t\wedge\sigma_n}\int_{{0}}^s\e^{2ru}\langle X(u),\Sigma(X_u)\d W(u) \rangle \right)\\
				&\leq \|\varphi(0)\|_H^2+C\e^{2rt}+ (\alpha+1+2L+2r)\E\int_{{0}}^{t\wedge\sigma_n} \e^{2rs} \|X(s)\|_H^2\d s\\
                &\quad+(\beta +2L)\E \int_{{0}}^{t\wedge\sigma_n} \int_{-\infty}^{{0}}\e^{2rs}\|X(s+\theta)\|_H^2 \rho(\d \theta)\d s\\
    &\quad+ 2\E\left(\sup _{0< s\leq t\wedge\sigma_n}\int_{{0}}^s\e^{2ru}\langle X(u),\Sigma(X_u)\d W(u) \rangle \right).
			\end{aligned}
		\end{equation*}
	By means of BDG's inequality, it follows that
			\begin{align*}
				2\E&\left(\sup _{0< s\leq t\wedge\sigma_n}\int_{{0}}^s\e^{2ru}\langle X(u),\Sigma(X_u)\d W(u) \rangle \right)\\
				&\leq 8\sqrt{2}\E\left(\sup_{{0} < s \leq {t\wedge\sigma_n}} \e^{2rs}\|X(s)\|_H^2\int_{{0}}^{t\wedge\sigma_n}\e^{2rs}\|\Sigma(X_s)\|_{L_2}^2\d s \right)^{\frac{1}{2}}\\
				&\leq \frac{1}{2}H(t)+64\E\int_{{0}}^{t\wedge\sigma_n}\e^{2rs}\|\Sigma(X_s)\|_{L_2}^2\d s\\
				&\leq \frac{1}{2}H(t)+C\e^{2r t}+128L\E\int_{0}^{{t\wedge\sigma_n}}  \e^{2r s}\|X(s)\|_H^{2}\mathrm{d} s\\
                &\quad+128L\E\int_{0}^{{t\wedge\sigma_n}} \int_{-\infty}^{0} \e^{2r s}\|X(s+\theta)\|_H^{2} \rho(\mathrm{d} \theta) \mathrm{d} s.
			\end{align*}
	Moreover, by the Fubini theorem and the variable substitution technique, we deduce that
		\begin{equation*}\label{eq:zhuyaoanother}
			\begin{aligned}
				\int_{0}^{t\wedge\sigma_n} &\int_{-\infty}^{0} \e^{2r s}\|X(s+\theta)\|_H^{2} \rho(\mathrm{d} \theta) \mathrm{d} s \\
				&=\int_{0}^{t} \int_{-\infty}^{-s} \e^{-2 r\theta} \e^{2 r(s+\theta)}\|X(s+\theta)\|_H^{2} \rho(\mathrm{d} \theta) \mathrm{d} s+\int_{-t\wedge\sigma_n}^{0}\int_{0}^{t\wedge\sigma_n+\theta} \e^{2r(s-\theta)}\|X(s)\|_H^{2} \mathrm{d} s \rho(\mathrm{d} \theta) \\
				&\leq \delta_{2r}\left(\rho\right)\|\varphi\|_{r}^{2}t+\delta_{2r}\left(\rho\right) \int_{0}^{t\wedge\sigma_n}\e^{2r s}\|X(s)\|_H^2\mathrm{d}s.
			\end{aligned}
		\end{equation*}
	Combining the above calculations implies that
 \begin{equation*}\label{eq:fuzhi}
     \begin{aligned}
         H(t)
         &\lesssim \|\varphi(0)\|_H^2+\e^{2rt}+ \E\int_{{0}}^{t\wedge\sigma_n} \e^{2rs} \|X(s)\|_H^2\d s+\E \int_{{0}}^{t\wedge\sigma_n} \int_{\infty}^{{0}}\e^{2rs}\|X(s+\theta)\|_H^2 \rho(\d \theta)\d s\\
         &\lesssim 1+\E\int_{{0}}^{t\wedge\sigma_n} \e^{2rs} \|X(s)\|_H^2\d s\lesssim 1+\int_{{0}}^{t} H(s)\d s
     \end{aligned}
 \end{equation*}
which, together with \gron's inequality, arrives at $H(t)\leq C\e^{Ct}.$
		Choosing $t=T$ gives 
  \begin{equation}\label{Eq:1031:1}
      n^2\P(\sigma_n\leq T)\leq \|\varphi\|_r^2+C\e^{CT}.
  \end{equation}
		Note that the right side of \eqref{Eq:1031:1} is independent of $n$. Letting $n\to \infty$ yields 
		\[
		\limsup_{n \rightarrow \infty}\P(\sigma_n\leq T) = 0,
		\]
		which implies that Eq. \eqref{EQ:GENERAL:01} has a unique global mild solution  $X(t)$ on $[0, \infty)$ almost surely.
\end{proof}
   
\end{appendix}

\section*{Declarations}
\noindent Ethical approval: Not applicable.

\bigskip\noindent Funding:  Fubao Xi and Zuozheng Zhang were supported by the National Natural Science Foundation of China under Grant No. 12071031.

\bigskip\noindent Availability of data and materials: Data and material sharing are not applicable to this article, as no datasets or materials were generated or analyzed during the current (theoretical) study.

\bibliographystyle{plain}
%使得没有引用的参考文献能被一并显示
%\nocite{*}

%\bibliography{reference_SFPDEs_ergodicity}		

\bibliography{SFPDEs_ergodicity_new}							
\end{document}